\documentclass[11pt,a4paper]{article} 

\usepackage[french,english]{babel}	
\usepackage[applemac]{inputenc} 
\usepackage{enumerate}    
\usepackage{lmodern}			
\usepackage{dsfont}			
\usepackage{amsmath, amsthm, amssymb, amsthm, latexsym, bbm,bm,mathrsfs,mathtools}	
\usepackage[pdfborder={0 0 0}]{hyperref}
\usepackage[toc,page]{appendix} 

\usepackage{authblk}

\usepackage{geometry,color}
\usepackage[all]{xy} 
\usepackage{lscape}

\usepackage{epstopdf}

\usepackage{mathrsfs} 

\newcommand{\tr}{\operatorname{tr}}
\newcommand{\Tr}{\operatorname{Tr}}
\theoremstyle{plain}
\newtheorem{theorem}{Theorem}[section]
\newtheorem{corollary}[theorem]{Corollary}
\newtheorem{prop}[theorem]{Proposition}

\newtheorem{lemme}[theorem]{Lemma}
\newtheorem{definition}[theorem]{Definition}
\newtheorem{rem}[theorem]{Remark}
\newtheorem{ex}[theorem]{Example} 

\def\br{\begin{color}{red}}
\def\er{\end{color}}

\def\C{\mathbb{C}}

\def\N{\mathbb{N}}

\def\EE{\mathbb{E}}
\def\Wg{\mathrm{Wg}}
\def\df{\mathrm{df}}

\def\<{\langle}
\def\>{\rangle}

\def\lto{\longrightarrow}
\def\lmto{\longmapsto}

\def\ts{\otimes}

\def\M{\mathrm{Moeb}}
\def\Mp{\mathrm{Moeb}^{(2)}}

\title{Freeness of type $B$ and conditional freeness\\for random matrices}
\author[,1]{Guillaume C\'{e}bron\thanks{G.C. has been partly supported by the ERC advanced grant ``Noncommutative distributions in free probability", partly by the Project
MESA (ANR-18-CE40-006) and partly by the Project STARS (ANR-20-CE40-0008) of the French National Research Agency. G.C. wishes to thank Roland Speicher, for giving him the opportunity to co-organize the Mini-workshop \emph{Notions of freeness} in April 2015, during which the authors conceived of
the first ideas that led to the present work.}}
\author[2]{Antoine Dahlqvist}
\author[3]{Franck Gabriel \thanks{F.G. has been partly supported by the ERC SG CONSTAMIS.}}

\affil[1]{Institut de Math\'{e}matique de Toulouse; UMR5219; Universit\'{e} de Toulouse; CNRS; UPS, F-31062 Toulouse, FRANCE. }
\affil[2]{University of Sussex, School of Mathematical and Physical Sciences, BN1 9RH Brighton, United Kingdom.}
\affil[3]{\'Ecole Polytechnique F\'ed\'erale de Lausanne; Chair of Statistical Field Theory; 1015 Lausanne, SWITZERLAND. }

\begin{document}

\maketitle
\begin{abstract}The asymptotic freeness of independent unitarily invariant $N\times N$ random matrices holds in expectation up to $O(N^{-2})$. An already known consequence is the infinitesimal freeness in expectation. We put in evidence another consequence for unitarily invariant random matrices: the almost sure asymptotic freeness of type~$B$. As byproducts, we recover the asymptotic cyclic monotonicity, and we get the asymptotic conditional freeness. In particular, the eigenvector empirical spectral distribution of the sum of two randomly rotated random matrices converges towards the conditionally free convolution. 
We also show new connections between infinitesimal freeness, freeness of type $B$, conditional freeness, cyclic monotonicity and monotone independence.
Finally, we show rigorously that the BBP phase transition for an additive rank-one perturbation of a GUE matrix is a consequence of the asymptotic conditional freeness, and the arguments extend to the study of the outlier eigenvalues of other unitarily invariant ensembles.
\end{abstract}

\section{Introduction}

\subsection{Infinitesimal freeness for random matrices}\label{Intro:inf_freeness}

In various random matrix models, the convergence of averaged mixed moments is as fast as $O(N^{-2})$ when the dimension $N\rightarrow \infty$. In particular, this means that random matrices which are asymptotically free in the sense of Voiculescu~\cite{Voiculescu1991} are in fact freely independent up to an error of order $O(N^{-2})$, leading asymptotically to an occurrence of the \emph{infinitesimal freeness} of Belinschi and Shlyakhtenko \cite{Belinschi2012}, or equivalently of the related \emph{freeness of type $B$} of Biane, Goodman and Nica~\cite{Biane2003}. For example, the pioneer result of Thorbjornsen in~\cite{Thorbjornsen2000}, which shows that the convergence of mixed moments of independent GUE matrices is in $O(N^{-2})$, can be rephrased as asymptotic infinitesimal freeness of independent GUE matrices.

More generally, asymptotic infinitesimal freeness can be established from first-order expansions of proofs about asymptotic freeness of unitarily invariant random matrices. This first-order expansion is proven  in~\cite[Theorem~5.11]{Curran2011}, showing that \emph{bounded} deterministic matrices become asymptotically infinitesimal free \textit{in expectation} when randomly rotated by independent Haar unitary matrices. In~\cite{Shlyakhtenko2018}, using a different line of arguments, including concentration results on large unitary groups
, Shlyakhtenko gave a short elegant proof of  asymptotic infinitesimal freeness  in expectation of a family of unitarily invariant matrices from a tuple of \emph{finite rank matrices}, and,   proved besides, with  explicit Gaussian computations, the asymptotic infinitesimal freeness  \emph{in expectation} of a family of independent GUE matrices from a tuple of \emph{finite rank matrices}. This result has been generalized in three directions: Collins, Hasebe and Sakuma proved in~\cite{Collins2018} the \emph{almost sure} asymptotic infinitesimal freeness of a family of unitarily invariant matrices from a tuple of \textit{finite rank matrices}, which can also be stated asymptotically as \emph{cyclic monotone independence}; Dallaporta and F\'{e}vrier proved in~\cite[Appendix A]{Dallaporta2019} that a family of independent GUE matrices is asymptotically infinitesimal free \textit{in expectation} from a tuple of bounded deterministic matrices converging in $*$-distribution; and Au proved in~\cite{au2021finite} the asymptotic infinitesimal freeness \emph{in expectation} of a family of Wigner matrices with finite moments from a tuple of \emph{finite rank matrices}, and of a family of periodically banded GUE matrix from a tuple of \emph{finite rank matrices}. Finally, the asymptotic infinitesimal freeness \emph{in expectation} of two independent random matrices, at least one of them being unitarily invariant, is also a consequence of the general theory of \emph{surfaced free probability} of Borot, Charbonnier, Garcia-Failde, Leid and Shadrin in~\cite{Borot2021} (it corresponds to $(\tfrac{1}{2},1)$-freeness in \cite[Theorem 4.28]{Borot2021}).

We shall recall how the proof given by Biane in~\cite[Section 9]{Biane1998} about asymptotic freeness of randomly rotated random matrices is sufficiently robust to give freeness of order $O(N^{-2})$ for matrices which are bounded in non-commutative distribution. 
  This result is stated in~Theorem~\ref{Mainth}. A second proof is given in Section \ref{section:matricial_cumulants} by providing a novel expansion of the \emph{matricial cumulants}: this yields an explicit and complete expansion in powers of $N^{-2}$ of the mixed moments of independent and unitarily invariant random matrices from which we recover the freeness of order $O(N^{-2})$.
  
In Corollary~\ref{coro: infFree Expect}, we deduce from it that randomly rotated matrices of arbitrary rank yields asymptotically \emph{infinitesimal freeness} in expectation. Using classical concentration estimates, we also infer in~Theorem~\ref{th:free_type_B} that \emph{freeness of type $B$} occurs asymptotically \emph{almost surely} on certain ideals, opening the door to the use of infinitesimal freeness for almost sure results, as we explain now. 

\subsection{Conditional freeness for random matrices} 

For the definition of the different notions of independence, the reader is referred to Section \ref{sec:cFree}. In~\cite{Collins2018}, Collins, Hasebe and Sakuma introduced the \emph{cyclic monotone independence}, which appears to be a consequence of infinitesimal freeness, or of freeness of type $B$, when we restrict our attention to some ideals where the normalized trace is trivial (see \cite[Proposition 3.10]{Collins2018}). On the other hand, \emph{cyclic monotone independence} reduces immediately to the \emph{monotone independence} of Muraki~\cite{Muraki2001} when we look at the distribution with respect to some vector state (see Proposition~\ref{Prop:mon}). We complete the picture by putting the \emph{conditional freeness} of Bo\.{z}ejko, Leinart and Speicher~\cite{Bozejko1996} into the following diagram, where horizontal implication are obtained by restricting to ideals where the normalized trace is trivial, and vertical implications are obtained by looking at the distribution with respect to particular vector states.
$$
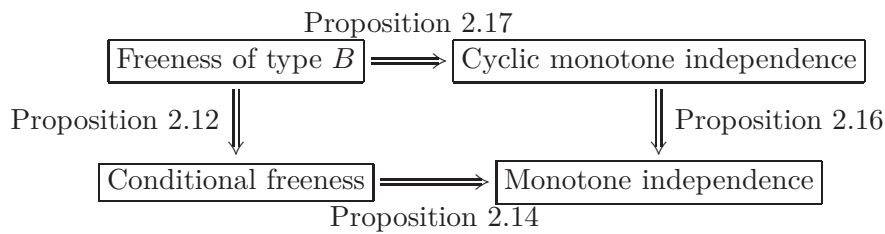
\xymatrix{
    **{[F]+}\text{Freeness of type }B \ar@{=>}[r]^-*++{\mbox{Proposition~\ref{from_typeB_to_mon}}} \ar@{=>}[d]_-*+{\mbox{Proposition~\ref{Prop:cfreeness}}} & **{[F]+}\text{Cyclic monotone independence} \ar@{=>}[d]^-*+{\mbox{Proposition~\ref{Prop:mon}}} \\
  **{[F]+}\text{Conditional freeness}  \ar@{=>}[r]_-*++{\mbox{Proposition~\ref{From_cfree_to_mon}}} & **{[F]+}\text{Monotone independence}
  }$$
The right vertical implication is quasi-immediate. As explained before, the upper horizontal implication is essentially contained in~\cite[Proposition 3.10]{Collins2018}. Also, the lower horizontal implication was already proved by Franz in~\cite[Proposition 3.1]{Franz2005}. The last implication (conditional freeness from freeness of type $B$) is new, even if other links between conditional freeness and freeness of type $B$ have been explored before (see~\cite{Belinschi2012} and~\cite{Fevrier2019}).

As freeness of type $B$ appears asymptotically for randomly rotated deterministic matrices, the diagram above allows us to produce various asymptotic independences for random matrices (both in expectation and almost surely). In particular, we recover in Theorem~\ref{th:cyclic_monotone} the result of \cite[Theorem 4.3]{Collins2018} about asymptotic cyclic monotone independence, we derive its corollary about asymptotic monotone independence in Corollary~\ref{th:mon}, and we exhibit a new phenomenon in~Theorem~\ref{th:conditionally_free}: randomly rotated deterministic matrices are asymptotically conditionally free. 

At the level of empirical spectral distribution, it means that conditional free convolution appears naturally in large dimension. More precisely, let $(A_N)_{N\geq 1}$ be a deterministic sequence  of $N\times N$ hermitian matrices, and $(v_N)_{N\geq 1}$ be a sequence of unit vectors of $\mathbb{C}^N$. The \emph{empirical spectral distribution} (ESD) of $A_N$ is 
$$\mu_{A_N}= \frac{1}{N} \sum_{\lambda \text{ eigenvalue of } A_N}\delta_{\lambda}$$
and the spectral distribution of $A_N$ with respect to the vector state given by $v_N$ is
$$\mu_{A_N}^{v_N}=\sum_{\substack{u_\lambda \text{unit-norm eigenvector of } A_N\\
\text{ associated to the eigenvalue }\lambda}}|\langle v_N,u_\lambda \rangle|^2\cdot \delta_{\lambda}.$$
The measure $\mu_{A_N}^{v_N}$ is sometimes called an \emph{eigenvector empirical spectral distribution} (VESD) of $A_N$, or the \emph{local density of the state} $v_N$.

In the following theorem, we use the free convolution $\mu_1\boxplus \mu_2$ and the conditionally free convolution $\nu_1\prescript{}{\mu_1}{\boxplus}^{}_{\mu_2} \mu_2$ which are defined in Section~\ref{Sec:convolutions}. Moreover, we say that a sequence of probability measures $(\mu_N)_{N\geq 1}$ converges to $\mu$ in moment whenever $\lim_N\int_{\mathbb{R}}x^kd \mu_N(x)=\int_{\mathbb{R}}x^kd \mu(x)$ for any $k\geq 0$.
\begin{theorem}\label{Th:sum}
Let $(A_N)_{N\geq 1}$ and $(B_N)_{N\geq 1}$ be two deterministic sequences  of $N\times N$ hermitian matrices (respectively real symmetric matrices), and $(v_N)_{N\geq 1}$ be a sequence of unit vectors of $\mathbb{C}^N$ (resp. of $\mathbb{R}^N$). We assume that almost surely, as $N$ tends to $\infty$, the following convergences hold in moment:
$$\mu_{A_N}\to \mu_1,\ \ \mu_{A_N}^{v_N}\to \nu_1\ \  \text{and}\ \ \mu_{B_N}\to \mu_2.$$
Let $(U_N)_{N\geq 1}$ be a sequence of Haar distributed unitary (resp. orthogonal) random matrices of size $N$. As $N$ tends to $\infty$, almost surely the following convergences hold in moment:
\begin{equation}\mu_{A_N+U_NB_NU_N^*}\to \mu_1\boxplus \mu_2\ \  \text{and}\ \ \mu_{A_N+U_NB_NU_N^*}^{v_N}\to \nu_1\prescript{}{\mu_1}{\boxplus}^{}_{\mu_2} \mu_2.\label{conv_cond_free}\end{equation}
%
\end{theorem}
While the first convergence to $\mu_1\boxplus \mu_2$ is the well-known asymptotic freeness of Voiculescu (see \cite{Speicher1993} for the almost sure convergence), the other convergence to $ \nu_1\prescript{}{\mu_1}{\boxplus}^{}_{\mu_2} \mu_2$ seems unknown. The Cauchy transform of $ \nu_1\prescript{}{\mu_1}{\boxplus}^{}_{\mu_2} \mu_2$ is given in Lemma~\ref{lemma_conv}. Note in particular that if $\mu_1=\delta_0$ (for example if $A_N$ has a bounded rank), then \eqref{conv_cond_free} reduces to the convergence
$$\mu_{A_N+U_NB_NU_N^*}^{v_N}\to \nu_1\prescript{}{\delta_0}{\boxplus}^{}_{\mu_2} \mu_2=\nu_1 \rhd \mu_2,$$
where $\rhd$ denote the monotone convolution of two measures.

Note that $\mu_{A_N+U_NB_NU_N^*}^{v_N}$ is the distribution of $A_N+U_NB_NU_N^*$ with respect to the state $\langle \cdot\ v_N,v_N\rangle$. More generally, if we are interested in the limit of $\mu_{P(A_N,U_NB_NU_N^*)}$ and $\mu_{P(A_N,U_NB_NU_N^*)}^{v_N}$ for an arbitrary polynomial $P\in \mathbb{C}\langle X,Y\rangle$ in two non-commuting variables such that $P(A_N,U_NB_NU_N^*)$ is self-adjoint, it is possible to use the asymptotic conditional freeness of $(A_N)_{N\geq 1}$ and $(U_NB_NU_N^*)_{N\geq 1}$  with respect to $\left(\frac{1}{N}\Tr,\langle \cdot\ v_N,v_N\rangle\right)$ in the following sense (see Theorem~\ref{th:conditionally_free} for more than two matrices).
\begin{theorem}\label{Th:poly}Under the assumptions of Theorem~\ref{Th:sum}, we have\begin{itemize}
\item for all $P\in \mathbb{C}\langle X,Y\rangle$, 
$ \frac{1}{N}\Tr(P(A_N, U_NB_NU_N^*))=\tau(P)+o(1)$ almost surely,
\item for all $P\in \mathbb{C}\langle X,Y\rangle$, 
$\langle P(A_N, U_NB_NU_N^*)v_N,v_N\rangle=\varphi(P)+o(1)$ almost surely,
\item $\mathbb{C}\langle X\rangle$ and $\mathbb{C}\langle Y\rangle$ are conditionally free with respect to $(\tau,\varphi)$.
\end{itemize}
\end{theorem}For example, the limiting eigenvector empirical spectral distribution of the product can be computed thanks to~\cite{popa2011multiplicative}.

At this stage, let us give an explanation of the usefulness of infinitesimal freeness, or of freeness of type $B$, in order to study the distribution with respect to the state $\langle \cdot\ v_N,v_N\rangle$. Because $\langle \cdot\ v_N,v_N\rangle=\Tr(v_Nv_N^* \cdot)$, the computation of $\langle \cdot\ v_N,v_N\rangle$ corresponds to the computation of the \emph{unnormalized} trace, which is made possible thanks to the asymptotic freeness of type $B$.

Finally, note the following consequence of Theorem~\ref{Th:poly}: if $\mu_1=\delta_0$ (for example if $A_N$ has a bounded rank), then $(A_N,B_N)$ is asymptotically monotone independent with respect to $\langle \cdot\ v_N,v_N\rangle$, as stated in Corollary~\ref{th:mon}.

\subsection{Related works}
If one is interested by the squared overlaps $|\langle v,u\rangle|^2$ between the eigenvectors $v$ of $A_N$ and the eigenvectors $u$ of $A_N+U_NB_NU_N^*$, it is useful to take $v_N$ as one of the eigenvectors of $A_N$, and study the local density $\mu_{A_N+U_NB_NU_N^*}^{v_N}$ of the state $v_N$,
as proposed by Allez and Bouchaud in~\cite{allez2014}.  If $\theta_N$ denotes the eigenvalue  associated to the eigenvector $v_N$ of $A_N$, the convergence of $\theta_N$ to $\theta$ is equivalent to the weak convergence of $\mu_{A_N}^{v_N}=\delta_{\theta_N}$ to a Dirac mass $\delta_\theta$, and our result \eqref{conv_cond_free} reduces to the convergence
$$\mu_{A_N+U_NB_NU_N^*}^{v_N}\to \delta_{\theta}\prescript{}{\mu_1}{\boxplus}^{}_{\mu_2} \mu_2.$$
In the case where $U_NB_NU_N^*$ is replaced by a matrix $X_N$ from the GUE, the measure $\mu_{X_N}$ converges to the semicircular distribution $\mu_{sc}$ as $N$ tends to $\infty$, and the weak convergence
$$\mathbb{E}\left[\mu_{A_N+X_N}^{v_N}\right]\to \delta_{\theta}\prescript{}{\mu_1}{\boxplus}^{}_{\mu_{sc}} \mu_{sc}$$
can also be deduced from \cite[Theorem 4.3]{allez2014} whenever $\theta \in supp(\mu_1)$. There exists a similar result in~\cite{ledoit2011} if one takes the product of $A_N$ with a Wishart matrix (see also the related works \cite{bun2016,bun2017,bun2018,bun2018optimal,BenaychGuionnetMaida,benaych2019empirical,potters2020,burda2021cleaning,benaych2022short}). As shown by Noiry in~\cite[Corollary 1]{Noiry2019}, this convergence holds more generally if $X_N$ is a Wigner matrix: 
as $N$ tends to $\infty$, we have in probability the weak convergence 
$$\mu_{A_N+X_N}^{v_N}\to \delta_{\theta}\prescript{}{\mu_1}{\boxplus}^{}_{\mu_{sc}} \mu_{sc}.$$
It is in fact a consequence of local laws estimates of Knowles and Yin~\cite{knowles2014outliers}. More generally, \emph{local laws estimates} for a random matrix model give information about the local density of states, and Theorem~\ref{Th:sum} can also be deduced from the local laws estimates for the model $A_N+U_NB_NU_N^*$ which are under study in \cite{Kargin2015} by Kargin, and also in a series of paper of Bao, Erd\H{o}s and Schnelli \cite{bao2016local,bao2017convergence,bao2017local,bao2020spectral}. See \cite{ding2020local} for a multiplicative analogue, and \cite{ho2022local} for a polynomial analogue. Alternatively, Theorem~\ref{Th:sum} is also an indirect consequence of~\cite[Proposition 5.3]{belinschi2017} in the case where $v_N$ is an eigenvector of $A_N$.

In the particular case where $A_N=0$, we have $\nu_1=\mu_1=\delta_0$, and the convergences~\eqref{conv_cond_free} reduce to the convergences
\begin{equation}\label{eq:equalityESDandVESD}\mu_{U_NB_NU_N^*}\to \mu_2\ \  \text{and}\ \ \mu_{U_NB_NU_N^*}^{v_N}\to \mu_2,
\end{equation}
which can be checked directly by a simple computation, using the unitary invariance. In this particular case, it is interesting to study the difference between the rate of convergence of the empirical spectral distribution $\mu_{U_NB_NU_N^*}$ (ESD) to $\mu_2$ and the rate of convergence of the eigenvector empirical spectral distribution $\mu_{U_NB_NU_N^*}^{v_N}$ (VESD) to $\mu_2$, not only for the random matrix $U_NB_NU_N^*$ but also for other random matrix models. In this direction, note the works of \cite{bai2007asymptotics,xia2013convergence,xia2019convergence,xi2020convergence} about quantitative estimates for the Kolmogorov distance and the works \cite{gamboa2011large,gamboa2016sum,gamboa2019sum,gamboa2022sum,noiry2020large} about large deviations estimates.
\subsection{Outlier eigenvalues} Studying explicitly the implication of infinitesimal freeness for empirical spectral measures, Shlyakthenko gave in~\cite{Shlyakhtenko2018} a beautiful potential interpretation of asymptotic infinitesimal freeness in terms of outliers of arbitrary polynomials in these matrices and finite rank perturbations. It raises the following question \cite[Problem 4.6]{Collins2018}:  is it possible, from the asymptotic infinitesimal freeness, to rigorously prove the BBP phase transition phenomena of outliers found by Baik, Ben Arous and P\'{e}ch\'{e} \cite{BBP,Peche}, and more generally by Benaych-Georges and Nadakuditi \cite{MR2782201,MR2944410} and Belinschi, Bercovici, Capitaine and F\'{e}vrier \cite{belinschi2017}?

We argue that the answer is positive: the almost sure location of the outliers in large dimension can be derived from the asymptotic freeness of type $B$ in the case of additive perturbation of arbitrary rank.  In order to enlighten the main idea of this derivation, let us recall the new and relatively simple approach of Noiry~\cite{Noiry2019} about outliers which allows to answer positively to this question in the particular case of a rank-one perturbation of a GUE matrix. In the following, we make use of the argument of Noiry and the asymptotic cyclic monotone independence of Collins, Hasebe and Sakuma \cite{Collins2018}, which is in some sense an almost sure version of asymptotic infinitesimal freeness.

Let $B_N$ be a GUE matrix. From~\cite{furedi1981}, we know that, not only does $\mu_{B_N}$ converge weakly to a probability measure $$\mu_{sc}=\frac{1}{2\pi}\sqrt{4-x^2}1_{|x|\leq 2}d x$$ called the semicircular law (Wigner's theorem), but also, the largest eigenvalue converges to $2$ almost surely and the least eigenvalue converges to $-2$ almost surely. Let $\theta \in \mathbb{R}$, $v_N\in \mathbb{C}^N$ be a unit vector and $A_N$ be the rank-one matrix $A_N=v_Nv_N^*$. We want to understand the outliers of the rank-one deformation
$B_N+\theta A_N.$

The almost sure asymptotic cyclic monotone independence of the pair $(A_N,B_N)$ is ensured by \cite[Theorem 4.3]{Collins2018}. One direct consequence (see e.g. Proposition~\ref{Prop:mon}) is the fact that $(A_N,B_N)$ is a pair which is asymptotically monotone independent with respect to the vector state $$M_N\mapsto \Tr(A_N M_N)=\langle M_N v_N,v_N\rangle.$$In particular, using Equation (\ref{eq:equalityESDandVESD}), we have almost surely the following weak convergence of the spectral distribution with respect to the vector state
\begin{equation}\lim_{N\to \infty}\mu^{v_N}_{B_N+\theta A_N}= \left(\lim_{N\to \infty}\mu^{v_N}_{\theta A_N}\right)\rhd\left(\lim_{N\to \infty}\mu^{v_N}_{B_N}\right)=\delta_\theta \rhd \mu_{sc}.\label{Noiry}\end{equation}
One can compute
\begin{equation*}
\delta_\theta \rhd \mu_{sc}=\frac{\sqrt{4-x^2}}{2\pi(\theta^2+1-\theta x)}1_{|x|\leq 2}d x +1_{|\theta|>1}\left(1-\frac{1}{\theta^2}\right)\delta_{\theta+\frac{1}{\theta}}\label{eq:LimitMeasureGUE}
\end{equation*}
by the Stieltjes-Perron inversion formula (see \cite{Biane1998b} or \cite{Noiry2019}). In fact, the convergence of $\mu^{v_N}_{B_N+\theta A_N}$ to $\delta_\theta \rhd \mu_{sc}$ is a result which has been first proved by Noiry~\cite[Proposition 2]{Noiry2019}. By conditioning on the spectrum of $B_N$, it is also a consequence of Theorem~\ref{Th:sum}, because $\delta_\theta \rhd \mu_{sc}=\delta_{\theta}\prescript{}{\delta_0}{\boxplus}^{}_{\mu_{sc}} \mu_{sc}$ by Lemma~\ref{lemma_conv}.

In order to conclude, we follow the argument of Noiry in \cite{Noiry2019}. The atom in the limiting distribution $\mu^{v_N}_{B_N+\theta A_N}$ means that, if $|\theta| >1$, there is at least one eigenvalue of $B_N+\theta A_N$ close to $\theta+\frac{1}{\theta}$ almost surely. Now, $A_N$ being of rank one, the eigenvalues of $B_N+\theta A_N$ interlace with the eigenvalues of $B_N$. More precisely, denoting by
$\lambda_1(\theta)\geq \lambda_2(\theta) \geq \ldots \geq \lambda_N(\theta)$
the eigenvalues of $B_N+\theta A_N$, we have
$$\lambda_1(\theta)\geq \lambda_1(0)\geq  \lambda_2(\theta) \geq \lambda_2(0) \geq \ldots \geq \lambda_N(\theta)\geq \lambda_N(0)$$
if $\theta >0$, and
$$\lambda_1(0)\geq \lambda_1(\theta)\geq  \lambda_2(0) \geq \lambda_2(\theta) \geq \ldots \geq \lambda_N(0)\geq \lambda_N(\theta)$$
if $\theta <0$. Because $\lim_{N\to \infty}\lambda_1(0)=2$ and $\lim_{N\to \infty}\lambda_N(0)=-2$ almost surely, we obtain as a consequence the following BBP phase transition  for rank-one deformation of a GUE matrix, originally due to \cite{Peche} (discovered in \cite{BBP} for Wishart matrices):
\begin{itemize} 
 \item If $\theta> 1$, there is one and only one outlier, which converges to $\theta+\frac{1}{\theta}>2$:
 
$\displaystyle \lim_{N\to \infty}\lambda_1(\theta)=\theta+\frac{1}{\theta}$, $\displaystyle \lim_{N\to \infty}\lambda_2(\theta)=2$ and $\displaystyle \lim_{N\to \infty}\lambda_N(\theta)=-2$ almost surely. 
  \item If $\theta< -1$, there is one and only one outlier, which converges to $\theta+\frac{1}{\theta}<-2$:
  
  $\displaystyle \lim_{N\to \infty}\lambda_1(\theta)=2$, $\displaystyle \lim_{N\to \infty}\lambda_{N-1}(\theta)=-2$ and $\displaystyle \lim_{N\to \infty}\lambda_N(\theta)=\theta+\frac{1}{\theta}$ almost surely.
\end{itemize}
 Moreover, denoting by $u_N$ the eigenvector associated to the outlier eigenvalue, we have  almost surely
$$ \lim_{N\to \infty} |\< u_N, v_N\>|^2= 1-\frac 1{\theta^2}$$
which was originally proved in~\cite{MR2782201}. For $|\theta|\leq 1$, we remark that the functions $\theta\mapsto  \lambda_1(\theta)$ and $\theta\mapsto  \lambda_N(\theta)$ are non-decreasing, and obtain
\begin{itemize} 
 \item 
 If $|\theta|\leq 1$, there are no outliers:
 
$\displaystyle \lim_{N\to \infty}\lambda_1(\theta)=2$ and $\displaystyle \lim_{N\to \infty}\lambda_N(\theta)=-2$ almost surely.
\end{itemize}
Note that the proof of \cite[Theorem 4.3]{Collins2018} is done by the method of moment, which means that we just derived rigorously the location of outliers by moment method. 
The previous argument of Noiry carries over to the general case of unitarily invariant random matrices, using~\eqref{conv_cond_free} instead of \eqref{Noiry}. We recover the following result for outliers of deformation of unitary invariant matrices, proved by Benaych-Georges and Nadakuditi in~\cite{MR2782201} for the case of finite rank deformation, and by Belinschi, Bercovici, Capitaine and F\'{e}vrier in~\cite{belinschi2017} for the general case (see also~\cite{Capitaine2018} for a general review). We need to use the subordination function $\omega_1$ of $\mu_1$ with respect to $\mu_1\boxplus \mu_2$, defined in Section~\ref{Sec:convolutions}.
\begin{theorem}\label{Th:outliers}
Let $(A_N)_{N\geq 1}$ and $(B_N)_{N\geq 1}$ be two deterministic sequences  of $N\times N$ hermitian matrices (respectively real symmetric matrices). We assume that almost surely, as $N$ tends to $\infty$, the following convergences hold in moment:
$$\mu_{A_N}\to \mu_1\ \   \text{and}\ \ \mu_{B_N}\to \mu_2,$$
where $\mu_1$ and $\mu_2$ are compactly supported. We assume moreover the existence of an \emph{outlier eigenvalue} of $A_N$ in the following sense: a sequence of eigenvalues $\theta_N\in Sp(A_N)$ such that, almost surely,
$$\lim_{N\to \infty}\theta_N= \theta\notin supp(\mu_1).$$
Let $(U_N)_{N\geq 1}$ be a sequence of Haar distributed unitary (resp. orthogonal) random matrices of size $N$. Then, for all $\rho \notin supp(\mu_1\boxplus \mu_2)$ such that $\omega_1(\rho)=\theta$,
\begin{itemize}
\item we have an \emph{outlier eigenvalue} of $A_N+U_NB_NU_N^*$ around $\rho$ in the following sense: there exists a sequence of eigenvalues $\rho_N\in Sp(A_N+U_NB_NU_N^*)$ such that, almost surely,
$$\lim_{N\to \infty} \rho_N=\rho\notin supp(\mu_1\boxplus \mu_2);$$
\item for all open intervals $I \subset \mathbb{R}\setminus supp(\mu_1\boxplus \mu_2)$ containing $\rho$, denoting by $v_N$ an eigenvector of $A_N$ associated to $\theta$ and by $E_N$ the set of eigenvectors of $A_N+U_NB_NU_N^*$ associated to eigenvalues in $I$, we have almost surely
$$\lim_{N\to \infty}\sum_{u\in E_N}\left|\langle u,v_N\rangle \right|^2 = \frac{1}{\omega'_1(\rho)}.$$
\end{itemize} 
\end{theorem}


\subsection{Organization of the paper}
The paper is organized as follows. In Section \ref{sec:cFree}, we discuss the algebraic framework that will later be useful in the paper. Several notions of freeness are reviewed, and we clarify and further study the connections among them. In Subsection \ref{Sec:convolutions}, we present the analytic tools used to compute convolutions for the different notions of freeness, namely the free, conditional free, and monotone convolutions.

In Section \ref{sec:asympfree}, we state the freeness property up to an $O(N^{-2})$ error that yields infinitesimal freeness of independent and randomly rotated matrices. This result is generalized in Section  \ref{sec:a.s.freeB} so as to provide an almost-sure version, stated in terms of asymptotic freeness of type B. Subsequently, various consequences of the algebraic framework developed in the previous section are stated. Mainly, in Theorem \ref{th:cond_freeness}, we obtain our new main corollary: unitarily invariant random matrices are asymptotically conditionally free (for the normalized trace and a deterministic state).

As a result of this new phenomenon, we can prove in Section \ref{sec:eigenvalues} the theorems from the introduction on the spectral distribution of the sum of two independent unitarily invariant matrices (Theorem \ref{Th:sum}), and the location of outliers in additive perturbations of a unitary invariant matrix by a matrix having an outlier (Theorem \ref{Th:outliers}).

In Section \ref{proof_of_Mainth}, we provide two proofs of the asymptotic freeness up to order $O(N^{-2})$ of independent and randomly rotated matrices (Theorem \ref{Mainth}). The first one, in Subsection \ref{sec:Weingarten}, relies on the second-order expansion of the Weingarten function, whereas the second one, in Subsection \ref{section:matricial_cumulants}, is obtained by defining a new expansion of the matricial cumulants, which provides a complete expansion of the mixed moments (Theorem \ref{Mainth-general}).

In Section \ref{Sec:orth}, we discuss the robustness of the results when the unitary invariance is replaced by the invariance by orthogonal conjugation. 

\newpage
\setcounter{tocdepth}{2}
\tableofcontents

\newpage
\section{Notions of freeness \label{sec:cFree}}In this section, we review the following different notions of freeness and some relations between them: freeness, infinitesimal freeness, freeness of order $o(t)$, freeness of type $B$, conditional freeness, cyclic monotone independence and monotone independence. We end the section with the definition of the corresponding convolutions of measures, and the use of the subordination functions to compute them.
\subsection{Freeness and infinitesimal freeness}
The following notion of freeness was introduced by Voiculescu in~\cite{Voiculescu1985}.
\begin{definition}\label{Def:freeness}A pair $(\mathcal{A},\tau)$ consisting of a unital algebra $\mathcal{A}$ and a linear functional $\tau : \mathcal{A}\to \mathbb{C}$ which is unital (i.e. $\tau(1)=1$) will be called a \emph{non-commutative probability space}.
One says that two unital subalgebras $\mathcal{A}_1,\mathcal{A}_2\subset \mathcal{A}$ are \emph{free with respect to $\tau$} if, whenever $a_1\in \mathcal{A}_{i_1},\ldots,a_n\in \mathcal{A}_{i_n}$ are such that $i_1\neq  i_2\neq i_3\neq\cdots$, and  $\tau(a_1)=\cdots=\tau(a_n)=0$:
\begin{equation}\tau(a_1\cdots a_n)=0.\label{Eq:freeness}\end{equation}
\end{definition}

Thanks to \cite[Lemma 5.18]{Nica2006}, the condition \eqref{Eq:freeness} on $\tau$ is equivalent to the following condition: whenever $a_n\in \mathcal{A}_{i_n},\ldots,a_1\in \mathcal{A}_{i_1},b_1\in \mathcal{A}_{j_1},\ldots,b_m\in \mathcal{A}_{j_m}$ are such that $i_1\neq i_2\neq i_3\neq\cdots$,
$j_1\neq  j_2\neq j_3\neq\cdots$, and  $\tau(a_n)=\cdots=\tau(a_1)=0=\tau(b_1)=\cdots=\tau(b_m)$:
\begin{equation}
\tau(a_n\cdots a_1b_1\cdots b_m)=\tau(a_nb_m)\cdots \tau(a_1b_1)\label{Eq:freeness_bis}\tag{\ref{Eq:freeness}'}
\end{equation}
if $n=m$ and $i_1=j_1,\ldots ,i_n=j_m$, and
$\tau(a_n\cdots a_1b_1\cdots b_m)=0$
otherwise.

Free independence appears asymptotically for certain random matrices. It can be expressed as freeness up to a certain order, accordingly to the following definition.
\begin{definition}Let $\mathcal{A}$ be a unital algebra generated by two unital subalgebras $\mathcal{A}_1,\mathcal{A}_2$. Let $T\subset \mathbb{R}$ be an index set such that $0$ is an accumulation point of $T$. Consider a family  $(\mathcal{A},\tau_t)_{t\in T}$ of non-commutative probability spaces.

We say that $\mathcal{A}_1$ and $\mathcal{A}_2$ are \emph{free with respect to $\tau_t$ up to order $o(t)$} if, whenever $a_1\in \mathcal{A}_{i_1},\ldots,a_n\in \mathcal{A}_{i_n}$ are such that $i_1\neq  i_2\neq i_3\neq \cdots$:
$$
\tau_t\Big((a_1-\tau_t(a_1))\cdots (a_n-\tau_t(a_n))\Big)=o(t).
$$
If we have the stronger conditions
$$\tau_t\Big((a_1-\tau_t(a_1))\cdots (a_n-\tau_t(a_n))\Big)=O(t^2),$$
we say that 
$\mathcal{A}_1$ and $\mathcal{A}_2$ are \emph{free with respect to $\tau_t$ up to order $O(t^2)$}.
\end{definition}

We might be given a family of unital linear functionals $\tau_t:\mathcal{A}\to \mathbb{C}$ so that $\tau_t=\tau+t\tau'+o(t)$ (for example $\tau_t=\tau+t\tau'$). Then, the infinitesimal freeness of $\mathcal{A}_1$ and $\mathcal{A}_2$ with respect to $(\tau,\tau')$ has been introduced in~\cite{Belinschi2012} as the freeness of $\mathcal{A}_1$ and $\mathcal{A}_2$ with respect to $\tau_t$ up to order $o(t)$. As explained in~\cite[Section 2.2]{Belinschi2012}, the previous requirement translates into the following definition.
\begin{definition}\label{Def:inf_freeness}
An  \emph{infinitesimal  non-commutative  probability  space} $(\mathcal{A},\tau,\tau')$ consists in a non-commutative probability space $(\mathcal{A},\tau)$ together with a linear functional $\tau' :\mathcal{A}\to \mathbb{C}$ with $\tau'(1)=0$. One says that two unital subalgebras $\mathcal{A}_1,\mathcal{A}_2\subset \mathcal{A}$ are \emph{infinitesimally free with respect to $(\tau,\tau')$} if, whenever $a_1\in \mathcal{A}_{i_1},\ldots,a_n\in \mathcal{A}_{i_n}$ are such that $i_1\neq i_2\neq i_3\neq\cdots$, and  $\tau(a_1)=\cdots=\tau(a_n)=0$:
\begin{align}
\tau(a_1\cdots a_n)&=0\notag,\\
\text{and}\ \ \ \tau' (a_1\cdots a_n)&=\sum_{j=1}^n\tau(a_1\cdots a_{j-1}\tau' (a_j )a_{j+1}\cdots a_n).\label{condition_inf_freeness}
\end{align}
\end{definition}
As remarked in~\cite[Remark 2.8.]{Fevrier2010}, the use of~\eqref{Eq:freeness_bis} in order to compute the r.h.s. of (\ref{condition_inf_freeness}) allows to replace this condition on $\tau'$ by the equivalent condition
\begin{equation}
\tau' (a_1\cdots a_n)=\tau(a_1a_n)\tau(a_2a_{n-1})\cdots \tau(a_{(n-1)/2} a_{(n+3)/2})\tau'(a_{(n+1)/2})\label{condition_inf_freeness_two}\tag{\ref{condition_inf_freeness}'}
\end{equation}
if $n$ is odd with $i_1=i_n$, $i_2=i_{n-1}$, $\ldots$, $i_{(n-1)/2}=i_{(n+3)/2}$, and
$\tau' (a_1\cdots a_n)=0$ otherwise. The link with freeness up to order $o(t)$, essentially contained in \cite[Section 2.2]{Belinschi2012}, is as follows.

\begin{prop}
 \label{Freeness_order_one}Let $\mathcal{A}$ be a unital algebra generated by two unital subalgebras $\mathcal{A}_1,\mathcal{A}_2$.  Let $T\subset \mathbb{R}$ be an index set such that $0$ is an accumulation point of $T$. Consider a family  $(\mathcal{A},\tau_t)_{t\in T}$ of non-commutative probability spaces  such that:
 \begin{itemize}
 \item $\mathcal{A}_1$ and $\mathcal{A}_2$ are free with respect to $\tau_t$ up to order $o(t)$,
 \item there exists $\tau' :\mathcal{A}_1\cup \mathcal{A}_2 \to \mathbb{C}$ such that for all $a\in \mathcal{A}_1 \cup \mathcal{A}_2$, as $t\to 0$, we have $$\tau_t(a)= \tau(a)+t\cdot \tau'(a) +o(t).$$
 \end{itemize}
 \noindent Then there exists an extension of $\tau'$ on $\mathcal{A}$, denoted also by $\tau'$, such that for all $a\in \mathcal{A}$,  as $t\to 0$,  $\tau_t(a)= \tau(a)+t\cdot \tau'(a) +o(t)$. Moreover, $\mathcal{A}_1$ and $\mathcal{A}_2$ are infinitesimally free with respect to $(\tau,\tau')$.
\end{prop}

\begin{proof}Usually, the infinitesimal freeness is shown from freeness of order $o(t)$ assuming the existence of $\tau$ and $\tau'$ on $\mathcal{A}$. Here, the expansion $\tau_t= \tau+t\cdot \tau'+o(t)$ on the whole algebra $\mathcal{A}$ is not in the assumptions, and for sake of completeness, we provide a full proof of the proposition.

For the existence of $\tau$ and $\tau'$ on $\mathcal{A}$, note that $\mathcal{A}$ is the linear span of $\mathcal{A}_1$, of $\mathcal{A}_2$ and of elements of the form $a_1\cdots a_n$ where $a_1,\ldots,a_n$ are as in Definition~\ref{Def:inf_freeness}. The existence of $\tau$, $\tau'$ on $\mathcal{A}_1$ and $\mathcal{A}_2$ as  $\tau_t= \tau+t\cdot \tau'+o(t)$ is part of the assumptions. Thanks to \eqref{condition_inf_freeness}, it suffices to show that 
\begin{equation}\label{eq:proof_inf_freeness}
\tau_t(a_1\cdots a_n)=\sum_{j=1}^nt\cdot \tau(a_1\cdots a_{j-1}\tau' (a_j )a_{j+1}\cdots a_n)+o(t)
\end{equation} for all $a_1,\ldots,a_n$ as in Definition~\ref{Def:inf_freeness} in order to show at the same time the existence of $\tau$, $\tau'$ on $\mathcal{A}$ as  $\tau_t= \tau+t\cdot \tau'+o(t)$, and the infinitesimal freeness of $\mathcal{A}_1$ and $\mathcal{A}_2$. Indeed, if Equation \eqref{eq:proof_inf_freeness} holds for any $n\leq n_0 $ ($n_0$ being a fixed integer) and any $a_1,\ldots,a_n$ as in Definition~\ref{Def:inf_freeness}, then, for any $k\leq n_0 $ and any $a_1,\ldots,a_k$ in $\mathcal{A}_1\cup \mathcal{A}_2$, there exists $\tau(a_1\ldots a_k)$ and $\tau'(a_1\ldots a_k)$ such that 
\begin{equation}
\label{eq:corollary_hypothesis_inf_freeness}\tau_t(a_1\cdots a_k)= \tau(a_1\ldots a_k)+t \cdot \tau'(a_1\ldots a_k)+o(t).
\end{equation}
We now proceed to prove Equation (\ref{eq:proof_inf_freeness}) by induction on $n$. The case $n=1$ is true by assumption. If $n>1$, let $a_1\in \mathcal{A}_{i_1},\ldots,a_n\in \mathcal{A}_{i_n}$ such that $i_1\neq i_2\neq i_3\neq \cdots$, and  $\tau(a_1)=\cdots=\tau(a_n)=0$. By freeness up to order $o(t)$, we have that
\begin{align*}
o(t)=&\tau_t\Big((a_1-\tau_t(a_1))\cdots (a_n-\tau_t(a_n))\Big)\\
=&\tau_t(a_1\cdots a_n) +\sum_{r=1}^n (-1)^r \sum_{1\leq k_1<\ldots<k_r\leq n} \left(\prod_{i=1}^{r} \tau_t(a_{k_i}) \right)\tau_t(a_1\cdots \hat{a}_{k_1}\cdots \hat{a}_{k_r} \cdots a_n)
\end{align*}
where $\hat{a}_i$ indicates terms that are omitted. For all $i$, $\tau_t(a_i)=\tau(a_i)+t\cdot \tau'(a_i)+o(t)=t\cdot \tau'(a_i)+o(t)$. With the induction hypothesis, and particularly its consequence, namely Equation (\ref{eq:corollary_hypothesis_inf_freeness}) for $k<n$, it yields that
$$\tau_t(a_{k_1})\cdots \tau_t(a_{k_r}) \tau_t(a_1\cdots \hat{a}_{k_1}\cdots \hat{a}_{k_r} \cdots a_n)= o(t)$$
whenever $2\leq r \leq n$, 
and
$$\tau_t(a_{k_1})\tau_t(a_1\cdots \hat{a}_{k_1}\cdots  a_n)=t\cdot \tau'(a_{k_1}) \tau(a_1\cdots \hat{a}_{k_1}\cdots  a_n)+o(t)$$
whenever $r=1$. Finally,
$$
o(t)=\tau_t(a_1\cdots a_n)+\sum_{1\leq k_1\leq n} -t\cdot \tau'(a_{k_1}) \tau(a_1\cdots \hat{a}_{k_1}\cdots  a_n)+o(t)$$
which allows to conclude:   
$\tau_t(a_1\cdots a_n)=\sum_{i=1}^nt\cdot \tau(a_1\cdots a_{j-1}\tau' (a_j )a_{j+1}\cdots a_n)+o(t).$
\end{proof}

In this article, we are mostly interested in infinitesimal freeness on a linear subspace of $\mathcal{A}$ which belongs to the kernel of $\tau$. In this case, most of the terms in \eqref{condition_inf_freeness} vanish and the formulae takes the following particular form.

\begin{lemme}\label{Inf_freeness_on_ideal}Let $(\mathcal{A},\tau,\tau')$ be an infinitesimal non-commutative probability space and let $\mathcal{A}_1,\mathcal{A}_2$ be two unital subalgebras which are infinitesimally free with respect to $(\tau,\tau')$. Consider an ideal $\mathcal{I}$ of $\mathcal{A}$ such that $\mathcal{I}\subset \ker(\tau)$. Set $\mathcal{I}_1:=\mathcal{I}\cap \mathcal{A}_1$ and $\mathcal{I}_2:=\mathcal{I}\cap \mathcal{A}_2$. 

Then,
whenever $a_n\in \mathcal{A}_{i_n},\ldots,a_1\in \mathcal{A}_{i_1},v\in \mathcal{I}_k,b_1\in \mathcal{A}_{j_1},\ldots,b_m\in \mathcal{A}_{j_m}$ are such that any two
consecutive indices in the list $i_n,\ldots,i_1,k,j_1,\ldots,j_m \in\{1,2\}$ are different  and  $\tau(a_n)=\cdots=\tau(a_1)=0=\tau(b_1)=\cdots=\tau(b_m)$, we have
\begin{equation}
\tau'(a_n\cdots a_1vb_1\cdots b_m)=\tau(a_n\cdots a_1\tau'(v)b_1\cdots b_m).\label{type_B_deux}
\end{equation}
\end{lemme}

As explained in the next section (see Proposition~\ref{B_freeness}), the condition \eqref{type_B_deux} is an occurrence of the so-called freeness of type $B$.
\subsection{Freeness of type $B$}
The infinitesimal freeness was originally introduced in~\cite{Belinschi2012} as a weakening of the
notion of freeness of type $B$. Non-commutative probability of type $B$ has been introduced in~\cite{Biane2003}, with a motivation which derives from lattice theory.

\begin{definition}A \emph{non-commutative probability space of type $B$} consists in a system $(\mathcal{A},\tau,V,\varphi,\Phi)$ where $(\mathcal{A},\tau)$ is a non-commutative probability space, $V$ is a complex vector space, $\varphi:V\to \mathbb{C}$ is a linear functional and $\Phi:(a_1,v,a_2)\mapsto a_1va_2$ is a two-sided action of $\mathcal{A}$ on $V$.
\end{definition}
In the case where $V$ is an ideal of $\mathcal{A}$, the two-sided action $\Phi$ is just the algebra action of $\mathcal{A}$ on itself. In this case, we shall drop the mention of $\Phi$ and write $(\mathcal{A},\tau,V,\varphi)$.
\begin{definition}\label{Def:freeness_B}Let $(\mathcal{A},\tau,V,\varphi,\Phi)$ be a non-commutative probability space of type $B$. Given two unital subalgebras $\mathcal{A}_1,\mathcal{A}_2\subset \mathcal{A}$ and two linear subspaces $V_1,V_2\subset V$ so that $V_1$ (respectively $V_2$) is invariant under the action of $\mathcal{A}_1$ (resp. of $\mathcal{A}_2$), one says that $(\mathcal{A}_1,V_1)$ and $(\mathcal{A}_2,V_2)$ are \emph{free of type $B$ in $(\mathcal{A},\tau,V,\varphi,\Phi)$} if $\mathcal{A}_1$ and $\mathcal{A}_2$ are
free with respect to $\tau$ and the following condition holds: whenever $a_n\in \mathcal{A}_{i_n},\ldots,a_1\in \mathcal{A}_{i_1},v\in V_h,b_1\in \mathcal{A}_{j_1},\ldots,b_m\in \mathcal{A}_{j_m}$ are such that any two
consecutive indices in the list $i_n,\ldots,i_1,h,j_1,\ldots,j_m $ are different  and  $\tau(a_n)=\cdots=\tau(a_1)=0=\tau(b_1)=\cdots=\tau(b_m)$, we have
\begin{equation}
\varphi (a_n\cdots a_1vb_1\cdots b_m)=\tau(a_nb_m)\cdots \tau(a_1b_1)\varphi(v)\label{type_B}
\end{equation}
if $n=m$ and $i_1=j_1,\ldots ,i_n=j_m$, and
$\varphi(a_n\cdots a_1vb_1\cdots b_m)=0$
otherwise.
\end{definition}
The use of~\eqref{Eq:freeness_bis} allows to replace the condition \eqref{type_B} on $\varphi$ by the equivalent condition
\begin{equation}
\varphi (a_n\cdots a_1vb_1\cdots b_m)=\tau(a_n\cdots a_1\varphi(v)b_1\cdots b_m)\label{type_B_bis}\tag{\ref{type_B}'}.
\end{equation}

Note that \eqref{type_B_bis} still applies when $a_n$ or $b_m$ is scalar. By bilinearity, it yields the following more general formulae for freeness of type $B$: whenever $a_n\in \mathcal{A}_{i_n},\ldots,a_1\in \mathcal{A}_{i_1},v\in V_h,b_1\in \mathcal{A}_{j_1},\ldots,b_m\in \mathcal{A}_{j_m}$ are such that any two
consecutive indices in the list $i_n,\ldots,i_1,h,j_1,\ldots,j_m $ are different  and  $\tau(a_{n-1})=\cdots=\tau(a_1)=0=\tau(b_1)=\cdots=\tau(b_{m-1})$, we have
\begin{equation}
\varphi (a_n\cdots a_1vb_1\cdots b_m)=\tau(a_n\cdots a_1\varphi(v)b_1\cdots b_m)\label{type_B_ter}\tag{\ref{type_B}''}.
\end{equation}

As explained in~\cite[Remark 2.9.]{Fevrier2010}, the similarities between condition~\eqref{condition_inf_freeness_two} and condition~\eqref{type_B} allows to incorporate the freeness of type B into the framework of infinitesimal freeness, by replacing the algebra by another one if necessary. The link between the two notions of freeness is even more direct in the case where $V$ is an ideal of $\mathcal{A}$ which is included in $\ker(\tau)$, as the following proposition show.

\begin{prop}\label{B_freeness}Let $(\mathcal{A},\tau,\tau')$ be an infinitesimal non-commutative probability space and let $\mathcal{A}_1,\mathcal{A}_2$ be two unital subalgebras which are infinitesimally free with respect to $(\tau,\tau')$.

Consider an ideal $\mathcal{I}$ of $\mathcal{A}$ such that $\mathcal{I}\subset \ker(\tau)$. Then $(\mathcal{A}_1,\mathcal{I}\cap \mathcal{A}_1)$ and $(\mathcal{A}_2,\mathcal{I}\cap \mathcal{A}_2)$ are free of type $B$ in the non-commutative probability space $(\mathcal{A},\tau,\mathcal{I},\tau')$ of type $B$.
\end{prop}

\begin{proof}The algebras $\mathcal{A}_1$ and $\mathcal{A}_2$ are free with respect to $\tau$, and a direct application of Lemma~\ref{Inf_freeness_on_ideal} yields~\eqref{type_B_bis}.
\end{proof}

In the case where $\tau'$ is not defined on the whole algebra, it is still possible to view freeness of type $B$ as a first-order expansion of freeness, as the following proposition shows (compare with Proposition~\ref{Freeness_order_one}).

\begin{prop}\label{Asymptotic_B_freeness}Let $\mathcal{A}$ be a unital algebra generated by two unital subalgebras $\mathcal{A}_1,\mathcal{A}_2$.  Let $T\subset \mathbb{R}$ be an index set such that $0$ is an accumulation point of $T$. Consider a family  $(\mathcal{A},\tau_t)_{t\in T}$ of non-commutative probability spaces and let  $\mathcal{I}_1$ (respectively $\mathcal{I}_2$) be an ideal of $\mathcal{A}_1$ (respectively of $\mathcal{A}_2$) such that:
\begin{itemize} 
\item $\mathcal{A}_1$ and $\mathcal{A}_2$ are free with respect to $\tau_t$ up to order $o(t)$, 
\item there exists $\tau: \mathcal{A}_1\cup \mathcal{A}_2 \to \mathbb{C}$ such that for $a\in \mathcal{A}_1\cup \mathcal{A}_2$, as $t\to 0$, $\tau_t(a)=\tau(a)+o(t),$
\item there exists $\tau': \mathcal{I}_1\cup \mathcal{I}_2 \to \mathbb{C}$ such that for $a \in \mathcal{I}_1\cup \mathcal{I}_2$, as $t\to 0$, $\tau_t(a)=t.\tau'(a)+o(t). $ 
\end{itemize}
Then, denoting by $\mathcal{I}$ the ideal of $\mathcal{A}$ generated by $\mathcal{I}_1\cup \mathcal{I}_2$, there exists an extension of $\tau$ on $\mathcal{A}$ and of $\tau'$ on $\mathcal{I}$, denoted also by $\tau$ and $\tau'$, such that, as $t\to 0$: 
\begin{itemize}
\item for all $a\in \mathcal{A}$,  $\tau_t(a)= \tau(a)+o(1)$, 
\item for all $a\in \mathcal{I}$, $ \tau_t(a)= t\cdot \tau'(a)+o(t)$.
\end{itemize}
Moreover, $(\mathcal{A}_1,\mathcal{I}_1)$ and $(\mathcal{A}_2,\mathcal{I}_2)$ are free of type $B$ in the non-commutative probability space $(\mathcal{A},\tau,\mathcal{I},\tau')$ of type $B$.
\end{prop}

\begin{rem}\label{Asymptotic_B_freeness_ortho}In the proof of Proposition~\ref{Asymptotic_B_freeness}, we do not use the full freeness  up to order $o(t)$ of $\mathcal{A}_1$ and $\mathcal{A}_2$, but we use it only on the ideal $\mathcal{I}$. In particular, the conclusion of the proposition still holds if this hypothesis is replaced by the following one:
\begin{itemize} 
\item the algebras $\mathcal{A}_1$ and $\mathcal{A}_2$ are free with respect to $\tau_t$ up to order $o(1)$, and whenever $a_n\in \mathcal{A}_{i_n}, \ldots,a_1\in \mathcal{A}_{i_1}$, $v\in \mathcal{I}_h,b_1\in \mathcal{A}_{j_1},\ldots,b_m\in \mathcal{A}_{j_m}$ be such that any two
consecutive indices in the list $i_n,\ldots,i_1,h,j_1,\ldots,j_m $ are different, we have
$$\tau_t\left((a_n-\tau_t(a_n))\cdots (a_1-\tau_t(a_1))(v-\tau_t(v))(b_1-\tau_t(b_1))\cdots (b_n-\tau_t(b_n))\right)=o(t).$$
\end{itemize}
\end{rem}
\begin{proof}
We remark that $\mathcal{A}$ is the linear span of $\mathcal{A}_1$, of $\mathcal{A}_2$ and of elements of the form $a_1\cdots a_n$ where $a_1,\ldots,a_n$ are as in Definition~\ref{Def:freeness}. The existence of $\tau$ on $\mathcal{A}_1$ and $\mathcal{A}_2$ as the limit of $\tau_t$ as $t\to 0$ is part of the assumptions. Thanks to \eqref{Eq:freeness}, it suffices to show that \begin{equation}
\tau_t(a_1\cdots a_n)=o(1)\label{eq:proof_freeness}
\end{equation} for all $a_1,\ldots,a_n$ as in Definition~\ref{Def:freeness} in order to show at the same time the existence of $\tau$ on $\mathcal{A}$ as  $\lim_t \tau_t$, and the freeness of $\mathcal{A}_1$ and $\mathcal{A}_2$.  Indeed, if Equation \eqref{eq:proof_inf_freeness} holds for any $n\leq n_0 $ ($n_0$ being a fixed integer) and any $a_1,\ldots,a_n$ as in Definition~\ref{Def:freeness}, then, for any $k\leq n_0 $ and any $a_1,\ldots,a_k$ in $\mathcal{A}_1\cup \mathcal{A}_2$, there exists $\tau(a_1\ldots a_k)$ and $\tau'(a_1\ldots a_k)$ such that 
\begin{equation}
\label{eq:corollary_hypothesis_freeness}\tau_t(a_1\cdots a_k)= \tau(a_1\ldots a_k)+o(1).
\end{equation}
We now proceed by induction on $n$. The case $n=1$ is true by assumption. If $n>1$, let $a_1\in \mathcal{A}_{i_1},\ldots,a_n\in \mathcal{A}_{i_n}$ such that $i_1\neq i_2\neq i_3\neq \cdots$ and  $\tau(a_1)=\cdots=\tau(a_n)=0$. As in the proof of Proposition \ref{Freeness_order_one}, using the freeness up to order $o(t)$, we have that
\begin{align*}
o(t)=\tau_t(a_1\cdots a_n) +\sum_{r=1}^n (-1)^r \sum_{1\leq k_1<\ldots<k_r\leq n} \left(\prod_{i=1}^{r} \tau_t(a_{k_i}) \right)\tau_t(a_1\cdots \hat{a}_{k_1}\cdots \hat{a}_{k_r} \cdots a_n)
\end{align*}
where $\hat{a}_i$ indicates terms that are omitted. For all $i$, $\tau_t(a_i)=o(1)$. With the induction hypothesis, and particularly its consequence, namely Equation (\ref{eq:corollary_hypothesis_freeness}) for $k<n$, it yields that
$$\tau_t(a_{k_1})\cdots \tau_t(a_{k_r}) \tau_t(a_1\cdots \hat{a}_{k_1}\cdots \hat{a}_{k_r} \cdots a_n)=o(1)$$
whenever $1\leq r \leq n$, 
which allows to conclude $\tau_t(a_1\cdots a_n)=o(1)$.

For the existence of $\tau'$ on $\mathcal{I}$, we remark that $\mathcal{I}$ is the linear span of elements of the form $a_n\cdots a_1vb_1\cdots b_m$ where $a_1,\ldots,a_n,v,b_1,\ldots,b_m$ are as in Definition~\ref{Def:freeness_B}. Thanks to \eqref{type_B_bis}, it suffices to show that \begin{equation}
\tau_t(a_n\cdots a_1vb_1\cdots b_m)=t\cdot \tau(a_n\cdots a_1\tau'(v)b_1\cdots b_m)+o(t)\label{eq:proof_Bfreeness}
\end{equation} for all such elements in order to show at the same time the existence of $\tau'$ on $\mathcal{I}$ as $\lim_t \frac{1}{t}\tau$ and the freeness of type $B$ of $(\mathcal{A}_1,\mathcal{I}_1)$ and $(\mathcal{A}_2,\mathcal{I}_2)$.   Indeed, if Equation \eqref{eq:proof_Bfreeness} holds for any $n+m\leq n_0 $ ($n_0$ being a fixed integer) and any $a_1,\ldots,a_n,v,b_1,\ldots,b_m$ as in Definition~\ref{Def:freeness_B}, then, for any $k+\ell\leq n_0 $, any $v\in \mathcal{I}_1\cup \mathcal{I}_2$ and any $a_1,\ldots,a_k,b_1,\ldots,b_\ell \in \mathcal{A}_1\cup \mathcal{A}_2$, there exists $\tau'(a_n\cdots a_1vb_1\cdots b_m)$ such that 
\begin{equation}
\label{eq:corollary_hypothesis_Bfreeness}\tau_t(a_n\cdots a_1vb_1\cdots b_m)= t\cdot \tau'(a_n\cdots a_1vb_1\cdots b_m)+o(t).
\end{equation}We now proceed by induction on $n+m$. The case $n+m=0$ is true by assumption. If $n+m>0$, let $a_n\in \mathcal{A}_{i_n},\ldots,a_1\in \mathcal{A}_{i_1},v\in V_h,b_1\in \mathcal{A}_{j_1},\ldots,b_m\in \mathcal{A}_{j_m}$ be such that any two
consecutive indices in the list $i_n,\ldots,i_1,h,j_1,\ldots,j_m $ are different  and  $\tau(a_n)=\cdots=\tau(a_1)=0=\tau(b_1)=\cdots=\tau(b_m)$.  
Using the notation 
 $$c_1=a_n,\ldots,c_n=a_1,c_{n+1}=v, c_{n+2}=b_1,\ldots,c_{n+m+1}=b_m, $$
by freeness up to order $o(t)$, we have $o(t)=\tau_t\left((c_1-\tau_t(c_1))\cdots (c_{n+m+1}-\tau_t(c_{n+m+1}))\right).$
As before, we develop the product and obtain that $\tau_t(c_1\cdots c_{n+m+1})$ is equal to: 
\begin{align*}
\sum_{r=1}^{n+m+1} (-1)^{r+1} \sum_{1\leq k_1<\ldots<k_r\leq n+m+1} \left(\prod_{i=1}^{r} \tau_t(c_{k_i}) \right)\tau_t(c_1\cdots \hat{c}_{k_1}\cdots \hat{c}_{k_r} \cdots c_{n+m+1})+o(t),
\end{align*}
where $\hat{c}_i$ indicates terms that are omitted.
For all $i$, $\tau_t(c_i)=o(1)$ and moreover $\tau_t(c_{n+1})=\tau_t(v)=t\cdot \tau'(v)+o(t)$. With the induction hypothesis, it yields that
$$\tau_t(c_{k_1})\cdots \tau_t(c_{k_r}) \tau_t(c_1\cdots \hat{c}_{k_1}\cdots \hat{c}_{k_r} \cdots c_{n+m+1})=o(t)$$
whenever $2\leq r$ or $k_1\neq n+1$, 
and 
$$\tau_t(c_{n+1}) \tau_t(c_1\cdots \hat{c}_{n+1} \cdots c_{n+m+1})=t \cdot \tau'(v)  \tau(c_1\cdots \hat{c}_{n+1} \cdots c_{n+m+1})+o(t).$$
This allows to conclude: $\tau_t(a_n\cdots a_1 v b_1 \cdots b_{m}) = t\cdot \tau(a_n\cdots a_1\tau'(v)b_1\cdots b_m) + o(t).$
\end{proof}
\subsection{Conditional freeness}
In a situation of free independence, $\tau$ is distributive over the product of alternating elements $a_1,\ldots, a_n$ which are centred with respect to $\tau$.  The conditional freeness with respect to $(\tau,\varphi)$,  introduced in~\cite{Bozejko1996}, is the fact that another linear functional $\varphi$ is also distributive over the product  of alternating elements $a_1,\ldots, a_n$ which are centered with respect to $\tau$.

\begin{definition}A \emph{conditional non-commutative probability space} $(\mathcal{A},\tau,\varphi)$ consists in a unital algebra $\mathcal{A}$ and two unital linear functionals $\tau, \varphi : \mathcal{A}\to \mathbb{C}$. One says that two unital subalgebras $\mathcal{A}_1,\mathcal{A}_2\subset \mathcal{A}$ are \emph{conditionally free with respect to $(\tau,\varphi)$} (or $c$-free) if, whenever  $a_1\in \mathcal{A}_{i_1},\ldots,a_n\in \mathcal{A}_{i_n}$ ($n>1$) are such that $i_1\neq  i_2\neq i_3\neq\cdots$, and  $\tau(a_1)=\cdots=\tau(a_n)=0$:
\begin{align}
\tau(a_1\cdots a_n)&=0,\notag\\
\varphi(a_1\cdots a_n)&=\varphi(a_1)\cdots \varphi(a_n).\label{eq:cond_freeness}
\end{align}
\end{definition}
 Note that \eqref{eq:cond_freeness} still applies when $a_1$ or $a_n$ is scalar. By bilinearity, it yields  the following more general formulae: whenever  $a_1\in \mathcal{A}_{i_1},\ldots,a_n\in \mathcal{A}_{i_n}$ ($n>1$) are such that $i_1\neq  i_2\neq i_3\neq\cdots$, and  $\tau(a_2)=\cdots=\tau(a_{n-1})=0$: 
\begin{align}
\tau(a_1\cdots a_n)&=0,\notag\\
\varphi(a_1\cdots a_n)&=\varphi(a_1)\cdots \varphi(a_n).\label{eq:cond_freeness_bis}\tag{\ref{eq:cond_freeness}'}
\end{align}

Connections between conditional freeness and freeness of type $B$ has been pointed out in~\cite{Belinschi2012} and in~\cite{Fevrier2019}. In the following proposition, we show that conditional freeness can also be recovered from freeness of type $B$ when looking at vector state in the direction of the kernel.
 \begin{prop}\label{Prop:cfreeness}
Let $\mathcal{A}$ be a unital algebra generated by two unital subalgebras $\mathcal{A}_1,\mathcal{A}_2$. Let $\mathcal{I}_1$ (respectively $\mathcal{I}_2$) be an ideal of $\mathcal{A}_1$ (respectively of $\mathcal{A}_2$) and $\mathcal{I}$ be the ideal of $\mathcal{A}$ generated by $\mathcal{I}_1\cup \mathcal{I}_2$. We consider a unital linear functional $\tau: \mathcal{A}\to \mathbb{C}$ such that $\mathcal{I}\subset \ker(\tau)$ and a linear functional $\tau': \mathcal{I}\to \mathbb{C}$ such that $(\mathcal{A}_1,\mathcal{I}_1)$ and $(\mathcal{A}_2,\mathcal{I}_2)$ are free of type $B$ in the non-commutative probability space $(\mathcal{A},\tau,\mathcal{I},\tau')$ of type $B$.

Let $p\in \mathcal{I}_1$ such that $\tau'(p)\neq 0$. Consider the conditional non-commutative probability space $(\mathcal{A},\tau,\varphi)$ where $\varphi:\mathcal{A}\to \mathbb{C}$ is defined by $$\varphi(a):=\frac{1}{\tau'(p)}\tau'(p a ).$$    Then, $\varphi$ coincides with $\tau$ on $\mathcal{A}_2$, and $\mathcal{A}_1,\mathcal{A}_2$ are conditionally free with respect to $(\tau,\varphi).$
 \end{prop}
 Note that in this particular case where $\varphi$ coincides with $\tau$ on $\mathcal{A}_2$, the condition \eqref{eq:cond_freeness} is completely equivalent to the fact that  whenever  $a_1\in \mathcal{A}_{i_1},\ldots,a_n\in \mathcal{A}_{i_n}$ ($n>1$) are such that $i_1\neq  i_2\neq i_3\neq\cdots$, and  $\tau(a_1)=\cdots=\tau(a_n)=0$:
$$
\varphi(a_1\cdots a_n)=0.$$
\begin{proof}For all $a\in \mathcal{A}_2$, using \eqref{type_B_ter}, we have
$$\varphi(a)=\frac{1}{\tau'(p)}\tau'(p a )=\frac{1}{\tau'(p)}\tau'(p)\tau( a )=\tau( a ),$$
hence  $\varphi$ coincides with $\tau$ on $\mathcal{A}_2$. 

Let us prove now the conditional freeness of $\mathcal{A}_1$ and $\mathcal{A}_2$ with respect to $(\tau,\varphi)$.
Let $a_1\in \mathcal{A}_{i_1},\ldots,a_n\in \mathcal{A}_{i_n}$ ($n>1$) such that $i_1\neq  i_2\neq i_3\neq\cdots$ and  $\tau(a_1)=\cdots=\tau(a_n)=0$. The freeness of $\mathcal{A}_1$ and $\mathcal{A}_2$ with respect to $\tau$ yields $\tau(a_1\cdots a_n)=\tau(a_2\cdots a_n)=0$. Furthermore, because $\tau$ coincides with $\varphi$ on $\mathcal{A}_2$, we have $\varphi(a_1)\cdots \varphi(a_n)=0$. The computation of $\varphi(a_1\cdots a_n)$ depends on whether $a_1$ belongs to $\mathcal{A}_1$ or to $\mathcal{A}_2$. 

\noindent If $a_1\in \mathcal{A}_1$, using \eqref{type_B_bis} with $pa_1\in \mathcal{I}_1$ playing the role of $v$, we have:
\begin{align*}
\varphi(a_1\cdots a_n)=\frac{1}{\tau'(p)}\tau'((p a_1) a_2\cdots a_n ) =\frac{1}{\tau'(p)}\tau'(pa_1)\tau(a_2\cdots a_n) =0.\end{align*}

\noindent If $a_1\in \mathcal{A}_2$, using \eqref{type_B_bis}, with $p\in \mathcal{I}_1$ playing the role of $v$, we have
\begin{align*}
\varphi(a_1\cdots a_n)&=\frac{1}{\tau'(p)}\tau'(p a_1\cdots a_n ) =\frac{1}{\tau'(p)}\tau'(p)\tau(a_1\cdots a_n) =0.\end{align*}
Hence, in both cases, $\varphi(a_1\cdots a_n)=0=\varphi(a_1)\cdots \varphi(a_n)$, which allows us to conclude.
\end{proof}
\subsection{Monotone independence}The monotone independence was introduced in~\cite{Muraki2001}, and originates from classification of notions of independence in non-commutative probability spaces. 

\begin{definition}Let $(\mathcal{A},\varphi)$ be a non-commutative probability space. We consider a (non-necessarily unital) subalgebra $\mathcal{A}_1\subset \mathcal{A}$ and a unital subalgebra $\mathcal{A}_2\subset \mathcal{A}$. We say that the pair $(\mathcal{A}_1,\mathcal{A}_2)$ is \emph{monotonically independent with respect to $\varphi$} if, whenever $a_1,\ldots,a_n\in \mathcal{A}_1$ and $b_0,b_1,\ldots,b_n\in \mathcal{A}_2$:
\begin{equation}
\varphi(b_0a_1b_1a_2b_2\cdots a_nb_n)=\varphi(a_1a_2\cdots a_n)\varphi(b_0)\varphi(b_1)\cdots \varphi(b_{n-1})\varphi(b_n).
\end{equation}
\end{definition}

The monotone independence can be recovered from conditional freeness, as shown in~\cite[Proposition 3.1]{Franz2005}. We state this relation in the following proposition.

\begin{prop}\label{From_cfree_to_mon}
Let $(\mathcal{A},\tau,\varphi)$ be a conditional non-commutative probability space and let $\mathcal{A}_1,\mathcal{A}_2$ be two unital subalgebras which are conditionnally free  with respect to $(\tau,\varphi)$.

We assume that $\varphi$ coincides with $\tau$ on $\mathcal{A}_2$. Consider an ideal $\mathcal{I}_1$ of $\mathcal{A}_1$ such that $\mathcal{I}_1\subset \ker(\tau)$. Then $(\mathcal{I}_1,\mathcal{A}_2)$ is monotonically independent with respect to $\varphi$.
\end{prop}
\begin{proof}If $a_1\in \mathcal{I}_1$ and $b_0,b_1\in \mathcal{A}_2$, we know that $\tau(a_1)=0$ because $\mathcal{I}_1\subset \ker(\tau)$. Thanks to \eqref{eq:cond_freeness_bis}, we have
\begin{equation}
\varphi(b_0a_1b_1)=\varphi(b_0)\varphi(a_1)\varphi(b_1).\label{two_terms_mon}
\end{equation}
Now, let $n> 1$, $a_1,\ldots,a_n\in \mathcal{I}_1$ and $b_0,c_1,\ldots,c_{n-1},b_n\in \mathcal{A}_2$ such that $\tau(c_1)=\ldots=\tau(c_{n-1})=0$. We also have $\tau(a_1)=\ldots = \tau(a_n)$ because $\mathcal{I}_1\subset \ker(\tau)$. We know thanks to~\eqref{type_B_ter} that
\begin{equation}\varphi(b_0a_1c_1a_2\cdots c_{n-1}a_nb_n)=\varphi(b_0)\varphi(a_1)\varphi(c_1)\varphi(a_2)\cdots \varphi(c_{n-1})\varphi(a_n)\varphi(b_n)=0.\label{more_than_two_terms_mon}\end{equation}
where we used that $\varphi(c_1)=\tau(c_1)=0.$

Let $n> 1$, $a_1,\ldots,a_n\in \mathcal{I}_1$ and $b_0,b_1,\ldots,b_n\in \mathcal{A}_2$, and set $c_i=b_i-\tau(b_i)$ for $1\leq i \leq n-1$. We compute
\begin{align*}
&\varphi(b_0a_1b_1a_2b_2\cdots a_nb_n)\\
=&\varphi(b_0a_1(c_1+\tau(b_1))a_2\cdots (c_{n-1}+\tau(b_{n-1}))a_nb_n)\\
=&\varphi(b_0a_1c_1a_2\cdots c_{n-1}a_nb_n)\\
&+\sum_{r=1}^n\sum_{1\leq k_1<\cdots<k_r\leq n-1}\varphi(b_0a_1c_1a_2\cdots \hat{c}_{k_1}\cdots \hat{c}_{k_r}\cdots c_{n-1}a_nb_n)\tau(b_{k_1})\cdots\tau(b_{k_r})
\end{align*}
where $\hat{c}_i$ indicates terms that are omitted. Using~\eqref{more_than_two_terms}, most of the terms vanish in the sum and we get
\begin{align*}
\varphi(b_0a_1b_1a_2b_2\cdots a_nb_n)&
=0+\varphi(b_0a_1a_2\cdots a_nb_n)\tau(b_1)\cdots \tau(b_{n-1})\\
&\varphi(b_0a_1a_2\cdots a_nb_n)\varphi(b_1)\cdots \varphi(b_{n-1}).\end{align*}
We conclude thanks to~\eqref{two_terms}:
$$\varphi(b_0a_1b_1a_2b_2\cdots a_nb_n)=\varphi(a_1a_2\cdots a_n)\varphi(b_0)\varphi(b_1)\cdots \varphi(b_{n-1})\varphi(b_n).
$$
\end{proof}
\subsection{Monotone cyclic independence}
The cyclic monotone independence was introduced in~\cite{Collins2018} as a replacement of monotone independence in order to describe the behaviour of some random matrices with respect to the non-normalized trace (see also the recent~\cite{Arizmendi2022} for an alternate point of view on cyclic monotone independence).
\begin{definition}Let $(\mathcal{C},\tau)$ be a non-commutative probability space. We consider a (non-necessarily unital) subalgebra $\mathcal{A}\subset \mathcal{C}$ and a unital subalgebra $\mathcal{B}\subset \mathcal{C}$. Set $$\mathcal{I}:=Span\{b_0a_1b_1\cdots a_nb_n:n\geq 1, a_1,\ldots,a_n\in \mathcal{A}, b_0,\ldots,b_n\in \mathcal{B}\}$$ and consider
a linear functional $\omega:\mathcal{I}\to \mathbb{C}$.
We say that the pair $(\mathcal{A},\mathcal{B})$ is \emph{cyclically monotone} with respect to $(\omega,\tau)$ if, whenever $a_1,\ldots,a_n\in \mathcal{A}$ and $b_0,b_1,\ldots,b_n\in \mathcal{B}$:
\begin{equation}
\omega(b_0a_1b_1a_2b_2\cdots a_nb_n)=\omega(a_1a_2\cdots a_n)\tau(b_1)\tau(b_2)\cdots \tau(b_{n-1})\tau(b_0b_n).\label{Eq:cyclic_monotone}
\end{equation}
\end{definition}

As conditional freeness is a consequence of freeness of type $B$ when looking at vector state (see Proposition~\ref{Prop:cfreeness}), monotone independence is a consequence of cyclic monotone independence when looking at vector state.

\begin{prop}\label{Prop:mon}
Let $(\mathcal{C},\tau)$ be a non-commutative probability space. We consider a (non-necessarily unital) subalgebra $\mathcal{A}\subset \mathcal{C}$ and a unital subalgebra $\mathcal{B}\subset \mathcal{C}$ such that $\mathcal{C}$ is generated by $\mathcal{A}\cup\mathcal{B}$. Set $$\mathcal{I}:=Span\{b_0a_1b_1\cdots a_nb_n:n\geq 1, a_1,\ldots,a_n\in \mathcal{A}, b_0,\ldots,b_n\in \mathcal{B}\}$$ and consider
a linear functional $\omega:\mathcal{I}\to \mathbb{C}$ such that $(\mathcal{A},\mathcal{B})$ is cyclically monotone with respect to $(\omega,\tau)$.

Let $p\in \mathcal{A}$ such that $\omega(p)\neq 0$. Consider the non-commutative probability space $(\mathcal{C},\varphi)$ where $\varphi:\mathcal{C}\to \mathbb{C}$ is defined by $$\varphi(a):=\frac{1}{\omega(p)}\omega(p a ).$$    Then, $\varphi$ coincides with $\tau$ on $\mathcal{B}$, and $(\mathcal{A},\mathcal{B})$ are monotonically independent with respect to $\varphi$.
 \end{prop}
\begin{proof}Let $b\in \mathcal{B}$. Thanks to~\eqref{Eq:cyclic_monotone}, we compute
$$\varphi(b)=\frac{1}{\omega(p)}\omega(p b )=\frac{1}{\omega(p)}\omega(p)\tau(b)=\tau(b)$$
and infer that
$\varphi$ coincides with $\omega$ on $\mathcal{B}$.

Let $a_1,\ldots,a_n\in \mathcal{A}$ and $b_0,b_1,\ldots,b_n\in \mathcal{B}$. Thanks to~\eqref{Eq:cyclic_monotone}, we compute
\begin{align*}
\varphi(b_0a_1b_1a_2b_2\cdots a_nb_n)&=\frac{1}{\omega(p)}\omega(pb_0a_1b_1a_2b_2\cdots a_nb_n)\\
&=\frac{1}{\omega(p)}\omega(pa_1\cdots a_n)\tau(b_0)\cdots \tau(b_n)\\
&=\varphi(a_1\cdots a_n)\varphi(b_0)\cdots \varphi(b_n).
\end{align*}
\end{proof}

Infinitesimal freeness yields  monotone cyclic independence if we look at the kernel of $\tau$, as shown in~\cite[Proposition 3.10]{Collins2018}. Similarly, freeness of type $B$ yields monotone cyclic independence, and we provide a proof below.

\begin{prop}Let $\mathcal{A}$ be a unital algebra generated by two unital subalgebras $\mathcal{A}_1,\mathcal{A}_2$. Let $\mathcal{I}_1$ (respectively $\mathcal{I}_2$) be an ideal of $\mathcal{A}_1$ (respectively of $\mathcal{A}_2$) and $\mathcal{I}$ be the ideal of $\mathcal{A}$ generated by $\mathcal{I}_1\cup \mathcal{I}_2$. We consider a unital linear functional $\tau: \mathcal{A}\to \mathbb{C}$ such that $\mathcal{I}\subset \ker(\tau)$ and a linear functional $\tau': \mathcal{I}\to \mathbb{C}$ such that $(\mathcal{A}_1,\mathcal{I}_1)$ and $(\mathcal{A}_2,\mathcal{I}_2)$ are free of type $B$ in the non-commutative probability space $(\mathcal{A},\tau,\mathcal{I},\tau')$ of type $B$.\label{from_typeB_to_mon}

Then, the pair $(\mathcal{I}_1,\mathcal{A}_2)$ is cyclically monotone with respect to $(\tau',\tau)$.
\end{prop}
By the symmetry of the situation, the pair $(\mathcal{I}_2,\mathcal{A}_1)$ is also cyclically monotone with respect to $(\tau',\tau)$ in the situation of the previous proposition.
\begin{proof}If $a_1\in \mathcal{I}_1$ and $b_0,b_1\in \mathcal{A}_2$, we know that $\tau(a_1)=0$ because $\mathcal{I}_1\subset \ker(\tau)$. Thanks to \eqref{type_B_ter}, we have
\begin{equation}
\tau'(b_0a_1b_1)=\tau(b_0\tau'(a_1)b_1).\label{two_terms}
\end{equation}
Now, let $n> 1$, $a_1,\ldots,a_n\in \mathcal{I}_1$ and $b_0,c_1,\ldots,c_{n-1},b_n\in \mathcal{A}_2$ such that $\tau(c_1)=\ldots=\tau(c_{n-1})=0$. We also have $\tau(a_1)=\ldots = \tau(a_n)$ because $\mathcal{I}_1\subset \ker(\tau)$. We know thanks to~\eqref{type_B_ter} that
\begin{equation}\tau'(b_0a_1c_1a_2\cdots c_{n-1}a_nb_n)=\tau(b_0\tau'(a_1)c_1a_2\cdots c_{n-1}a_nb_n)=0.\label{more_than_two_terms}\end{equation}
where we used that $b_0c_1a_2\cdots a_nb_n\in\mathcal{I}\subset \ker(\tau).$

Let $n> 1$, $a_1,\ldots,a_n\in \mathcal{I}_1$ and $b_0,b_1,\ldots,b_n\in \mathcal{A}_2$, and set $c_i=b_i-\tau(b_i)$ for $1\leq i \leq n-1$. We compute
\begin{align*}
&\tau'(b_0a_1b_1a_2b_2\cdots a_nb_n)\\
=&\tau'(b_0a_1(c_1+\tau(b_1))a_2\cdots (c_{n-1}+\tau(b_{n-1}))a_nb_n)\\
=&\tau'(b_0a_1c_1a_2\cdots c_{n-1}a_nb_n)\\
&+\sum_{r=1}^n\sum_{1\leq k_1<\cdots<k_r\leq n-1}\tau'(b_0a_1c_1a_2\cdots \hat{c}_{k_1}\cdots \hat{c}_{k_r}\cdots c_{n-1}a_nb_n)\tau(b_{k_1})\cdots\tau(b_{k_r})
\end{align*}
where $\hat{c}_i$ indicates terms that are omitted. Using~\eqref{more_than_two_terms}, most of the terms vanish in the sum and we get
$$
\tau'(b_0a_1b_1a_2b_2\cdots a_nb_n)
=0+\tau'(b_0a_1a_2\cdots a_nb_n)\tau(b_1)\tau(b_2)\cdots \tau(b_{n-1}).$$
We conclude thanks to~\eqref{two_terms}:
$$\tau'(b_0a_1b_1a_2b_2\cdots a_nb_n)=\tau'(a_1a_2\cdots a_n)\tau(b_1)\tau(b_2)\cdots \tau(b_{n-1})\tau(b_0b_n).
$$
\end{proof}
\subsection{Computation of convolutions}\label{Sec:convolutions}
Let $(\mathcal{A},\tau)$ be a probability space. If $\mathcal{A}$ is a $*$-algebra and $\tau$ a state, i.e. in addition to $\tau(1)=1$ also positive ($\tau(aa^*)\geq 0$ for all $a\in \mathcal{A}$), then we call $(\mathcal{A},\tau)$ a $*$-probability space. A law of a self-adjoint element $a=a^*\in \mathcal{A}$ (with respect to $\tau$) is any probability measure $\mu$ on $\mathbb{R}$, such that, for all $k\in \mathbb{R}$,
$$\int_\mathbb{R}x^k\mu(dx)=\tau(a^k). $$
When the moment sequence $(\tau(a^k))_{k\geq 1}$ is determinate,  we will speak about \emph{the law} $\mu$ of a self-adjoint element $a=a^*\in \mathcal{A}$ with respect to $\tau$ (for example if the support of $\mu$ is bounded).

We review now the different notions of convolutions. For simplicity, we first only consider probability measures with compact supports. Let $(\mathcal{A},\tau)$ be a $*$-probability space, and $a_1$ and $a_2$ be two self-adjoint elements of $\mathcal{A}$ with respective laws $\mu_1$ and $\mu_2$ (with respect to $\tau$) with compact supports. We denote by $\mathcal{A}_1$ the algebra generated by $a_1$, by $\mathcal{I}_1$ the ideal of $\mathcal{A}_1$ generated by $a_1$ and by $\mathcal{A}_2$ the algebra generated by $a_2$.

The \emph{free convolution} $\mu_1\boxplus\mu_2$ is the law of $a_1+a_2$  with respect to $\tau$ whenever $\mathcal{A}_1$ and $\mathcal{A}_2$ are free with respect to $\tau$. It is uniquely defined by the pair $(\mu_1,\mu_2)$.

The \emph{monotone convolution} $\mu_1\rhd \mu_2$ is the law of $a_1+a_2$  with respect to $\tau$ whenever $(\mathcal{I}_1,\mathcal{A}_2)$ is monotonically independent with respect to $\tau$. It is uniquely defined by the pair $(\mu_1,\mu_2)$.

Now consider another state $\varphi:\mathcal{A}\to \mathbb{C}$, and $\nu_1$ and $\nu_2$ be the respective laws of $a_1$ and $a_2$ (with respect to $\varphi$) with compact supports. The \emph{conditionally free convolution} $\nu_1\prescript{}{\mu_1}{\boxplus}^{}_{\mu_2} \nu_2$ is the law of $a_1+a_2$  with respect to $\varphi$ whenever $\mathcal{A}_1$ and $\mathcal{A}_2$ are conditionally free with respect to $(\tau,\varphi)$. It is uniquely defined by the measures $(\nu_1,\nu_2,\mu_1,\mu_2)$.

We review now the complex analytic characterizations of free convolutions, and extend the definition of $\mu_1\boxplus\mu_2$, $\mu_1\rhd\mu_2$ and $\nu_1\prescript{}{\mu_1}{\boxplus}^{}_{\mu_2} \nu_2$ for any probability measures (with possibly unbounded support), as in \cite{Bercovici1993,Muraki2000,Belinschi2008}. Let $\mathbb{C}^+$ be the set of complex numbers $z$ such that $im(z)>0$. The Cauchy transform $$ G_\mu(z) =\int_\mathbb{R}\frac{\mu(dx)}{z-x},\ \  z\in \mathbb{C}^+$$ of a probability measure $\mu$ plays a crucial role for the computation of free, monotone or conditionally free convolutions.  We define also its reciprocal $$F_\mu(z) =\frac{1}{G_\mu(z)},\ \  z\in \mathbb{C}^+.$$Let $\mu_1$ and $\mu_2$ be two probability measures on $\mathbb{R}$.  The following subordination properties are due to~\cite{Biane1998b,Bercovici1998} (see also~\cite[Theorem 3.43]{Mingo2018}). There exist two unique analytic functions $\omega_1,\omega_2:\mathbb{C}^+\to \mathbb{C}^+$ such that
$$F_{\mu_1}(\omega_1(z))=F_{\mu_2}(\omega_2(z))$$
and $$\omega_1(z)+\omega_2(z)=z+F_{\mu_1}(\omega_1(z)).$$
The function $\omega_1$ (respectively $\omega_2$) is called the \emph{subordination function} of $\mu_1$ (resp. of $\mu_2$) with respect to $\mu_1\boxplus\mu_2$. The analytic function $\omega_1$ is known to extend meromorphically to the complement $\mathbb{C}\setminus supp(\mu_1\boxplus\mu_2)$ of the support of $\mu_1\boxplus\mu_2$ (see \cite[Lemma 2.3]{bercovici2008freely}, or \cite[Lemma 3.1]{belinschi2017}). Moreover, $\omega_1$ is real-valued on $\mathbb{R}\setminus supp(\mu_1\boxplus\mu_2)$ and for all $x\in \mathbb{R} \setminus supp(\mu_1\boxplus\mu_2)$, we have
$$\omega_1(x)\in \mathbb{R}\cup \{\infty\} \setminus supp(\mu_1).$$
\begin{definition}
The \emph{free convolution} $\mu_1\boxplus\mu_2$ is the unique probability measure on $\mathbb{R}$ such that
$$F_{\mu_1\boxplus\mu_2}(z)=F_{\mu_1}(\omega_1(z))=F_{\mu_2}(\omega_2(z)).$$
The \emph{monotone convolution} $\mu_1\rhd\mu_2$ is the unique probability measure on $\mathbb{R}$ such that
$$F_{\mu_1\rhd\mu_2}(z)=F_{\mu_1}(F_{\mu_2}(z)).$$
Let $\nu_1$ and $\nu_2$ be two probability measures on $\mathbb{R}$. The \emph{conditionally free convolution} $\nu_1\prescript{}{\mu_1}{\boxplus}^{}_{\mu_2} \nu_2$ is the unique probability measure on $\mathbb{R}$ such that
$$F_{\nu_1\prescript{}{\mu_1}{\boxplus}^{}_{\mu_2} \nu_2}(z)=F_{\nu_1}(\omega_1(z))+F_{\nu_2}(\omega_2(z))-F_{\mu_1\boxplus \mu_2}(z). $$
\end{definition}
Whenever $\mu_2=\nu_2$, we have $F_{\nu_2}(\omega_2(z))=F_{\mu_2}(\omega_2(z))=F_{\mu_1\boxplus \mu_2}(z)$. In particular,
$$F_{\nu_1\prescript{}{\mu_1}{\boxplus}^{}_{\mu_2} \mu_2}(z)=F_{\nu_1}(\omega_1(z)).$$Moreover, if $\mu_1=\delta_0$, we have $\omega_1=F_{\mu_2}$. It implies the following result, which corresponds to Corollary 2.2 of \cite{Hasebe2010}.

\begin{lemme}\label{lemma_conv}For probability measures $\mu_1, \nu_1, \mu_2$ on $\mathbb{R}$, the conditionally free convolution $\nu_1\prescript{}{\mu_1}{\boxplus}^{}_{\mu_2} \mu_2$ is characterized by
$$G_{\nu_1\prescript{}{\mu_1}{\boxplus}^{}_{\mu_2} \mu_2}(z)=G_{\nu_1}(\omega_1(z)).$$
Moreover, if $\mu_1=\delta_0$,
$$G_{\nu_1\prescript{}{\delta_0}{\boxplus}^{}_{\mu_2} \mu_2}(z)=G_{\nu_1}(F_{\mu_2}(z))=G_{\nu_1\rhd\mu_2}(z).$$
In particular, $\nu_1\prescript{}{\delta_0}{\boxplus}^{}_{\mu_2} \mu_2=\nu_1\rhd\mu_2$.
\end{lemme}


%
%

\section{Asymptotic freeness of random matrices \label{sec:asympfree}}

From now on, the normalized trace $\frac{1}{N} \Tr$ is denoted by $\tr_N$.

\begin{definition}Let $\mathbf{A}_N$  be a  $k$-tuple $\mathbf{A}_N=(A_N^{(1)},\ldots,A_N^{(k)})$ of $N\times N$ matrices. The  \emph{non-commutative distribution} (or n.c. distribution) of $\mathbf{A}_N$ is the linear functional $$ \begin{array}{rcl}\tau_{\mathbf{A}_N}:\mathbb{C}\langle X_1,\ldots,X_k\rangle &\lto& \mathbb{C} \\
P&\lmto& \tr_N\left( P(A_N^{(1)},\ldots,A_N^{(k)})\right). \end{array}$$
If $\mathbf{B}_N$  is another  $\ell$-tuple $\mathbf{B}_N=(B_N^{(1)},\ldots,B_N^{(\ell)})$ of $N\times N$ matrices, their \emph{joint distribution} $\tau_{\mathbf{A}_N,\mathbf{B}_N}$ is the linear functional
$$ \begin{array}{rcl}\tau_{\mathbf{A}_N,\mathbf{B}_N}:\mathbb{C}\langle X_1,\ldots,X_k,Y_1,\ldots,Y_{\ell}\rangle &\lto& \mathbb{C} \\
P&\lmto& \tr_N\left( P(A_N^{(1)},\ldots,A_N^{(k)},B_N^{(1)},\ldots,B_N^{(\ell)})\right). \end{array}$$
The \emph{free product} $\tau_{\mathbf{A}_N}\star \tau_{\mathbf{B}_N}$ of their non-commutative distribution $\tau_{\mathbf{A}_N}$ and $\tau_{\mathbf{B}_N}$ is the unique linear functional on $\mathbb{C}\langle X_1,\ldots,X_k,Y_1,\ldots,Y_{\ell}\rangle$ such that $\tau_{\mathbf{A}_N}\star \tau_{\mathbf{B}_N}$ coincides with $\tau_{\mathbf{A}_N,\mathbf{B}_N}$ on $\mathbb{C}\langle X_1,\ldots,X_{k}\rangle \cup \mathbb{C}\langle Y_1,\ldots,Y_{\ell}\rangle$, and $\mathbb{C}\langle X_1,\ldots,X_{k}\rangle$ and $\mathbb{C}\langle Y_1,\ldots,Y_{\ell}\rangle$ are free with respect to $\tau_{\mathbf{A}_N}\star \tau_{\mathbf{B}_N}$. 
\end{definition}
\subsection{Asymptotic infinitesimal freeness}

\begin{definition}
  We call a sequence $(\mathbf{A}_N)_{N\geq 1}$ of $k$-tuple $\mathbf{A}_N=(A_N^{(1)},\ldots,A_N^{(k)})$ of $N\times N$  matrices \emph{bounded in non-commutative distribution} (or in n.c. distribution) if  for  all polynomials $P \in \mathbb{C}\langle X_1,\ldots,X_k\rangle$,
$$\sup_{N\ge 1} | \tr_N\left(P(\mathbf{A}_N\right) | <+\infty.$$
\end{definition}

We first state the freeness up to order $O(N^{-2})$ without any hypothesis of convergence for the family of random matrices.
\begin{theorem}\label{Mainth}
Let $(\mathbf{A}_N)_{N\geq 1}$ and $(\mathbf{B}_N)_{N\geq 1}$ be two deterministic sequences of $k$-tuples and $\ell$-tuples of $N\times N$ matrices, each of them being bounded in n.c. distribution.  Let $(U_N)_{N\geq 1}$ be a sequence of Haar distributed unitary  random matrices of size $N$.

Then, for all $P\in \mathbb{C}\langle X_1,\ldots,X_{k},Y_1,\ldots,Y_{\ell}\rangle$,
$$\EE[\tau_{\mathbf{A}_N, U_N\mathbf{B}_NU_N^*}(P)]=\tau_{\mathbf{A}_N}\star \tau_{\mathbf{B}_N}(P)+O(N^{-2}).$$ In other words, $\mathbb{C}\langle X_1,\ldots,X_{k}\rangle$ and $\mathbb{C}\langle Y_1,\ldots,Y_{\ell}\rangle$ are free up to order $O(N^{-2})$ with respect to $\EE[\tau_{\mathbf{A}_N, U_N\mathbf{B}_NU_N^*}]$. 
\end{theorem}

\begin{proof}
As explained in Section~\ref{Intro:inf_freeness}, the fact that the error term in the asymptotic freeness of unitarily invariant matrices is in $O(N^{-2})$ is part of the folklore in random matrices (see e.g. \cite[Lemma 4.3.2]{hiai2000semicircle}). We give two proofs of Theorem~\ref{Mainth} in Section~\ref{proof_of_Mainth}, the second one producing a complete expansion of $\EE[\tau_{\mathbf{A}_N, U_N\mathbf{B}_NU_N^*}(P)]$ in $N^{-2}$.
\end{proof}
%

If  $\mathbf{A}_N$ and $\mathbf{B}_N$ converge both in distribution and in infinitesimal distribution, the previous theorem implies the following asymptotic infinitesimal freeness \emph{in expectation} (see Section~\ref{Intro:inf_freeness} for bibliographic references about this result). 

\begin{corollary}\label{coro: infFree Expect}Let $(\mathbf{A}_N)_{N\geq 1}$ and $(\mathbf{B}_N)_{N\geq 1}$ be two deterministic sequences of $k$-tuples and $\ell$-tuples of $N\times N$ matrices, such that
\begin{itemize}
\item  for all $P\in \mathbb{C}\langle X_1,\ldots,X_k\rangle$, 
$\tau_{\mathbf{A}_N}(P)=\tau(P)+\frac{1}{N}\tau'(P)+o(N^{-1}),$
\item 
for all $P\in \mathbb{C}\langle Y_1,\ldots,Y_\ell\rangle$,
$\tau_{\mathbf{B}_N}(P)=\tau(P)+\frac{1}{N}\tau'(P)+o(N^{-1}).$
\end{itemize}
Let $(U_N)_{N\geq 1}$ be a sequence of Haar distributed unitary  random matrices of size $N$. Then,
\begin{itemize}
\item $\tau$ and $\tau'$ extend to $\mathbb{C}\langle X_1,\ldots,X_k,Y_1,\ldots,Y_{\ell}\rangle$ via
$$\EE[ \tau_{\mathbf{A}_N, U_N\mathbf{B}_NU_N^*}(P)]=\tau(P)+\frac{1}{N}\tau'(P)+o(N^{-1}),$$
\item $\mathbb{C}\langle X_1,\ldots,X_{k}\rangle$ and $\mathbb{C}\langle Y_1,\ldots,Y_{\ell}\rangle$ are infinitesimally free  with respect to $(\tau,\tau')$. 
\end{itemize}
\end{corollary}
\begin{proof}It is a direct application of Proposition~\ref{Freeness_order_one}.
\end{proof}
\subsection{Asymptotic freeness of type $B$ \label{sec:a.s.freeB}}
The following theorem is an almost sure version of the infinitesimal freeness of previous section. As we need to restrict ourselves to ideals, it is more convenient to express it thanks to freeness of type $B$.

\begin{theorem}
\label{th:asymp_freeness_type_B}
Let $0\leq k'\leq k$ and $0\leq \ell'\leq \ell$. We set $\mathcal{A}:= \mathbb{C}\langle X_1,\ldots,X_k,Y_1,\ldots,Y_{\ell}\rangle$, $\mathcal{A}_1=\mathbb{C}\langle X_1,\ldots,X_k\rangle$ and $\mathcal{A}_2=\mathbb{C}\langle Y_1,\ldots,Y_{\ell}\rangle$. We denote by $\mathcal{I}_1$ the ideal of $\mathcal{A}_1$ generated by $X_1,\ldots,X_{k'}$, by $\mathcal{I}_2$ the ideal of $\mathcal{A}_2$ generated by $Y_1,\ldots,Y_{\ell'}$ and by $\mathcal{I}$ the ideal of $\mathcal{A}$ generated by $\mathcal{I}_1\cup \mathcal{I}_2$.

Let $(\mathbf{A}_N)_{N\geq 1}$ and $(\mathbf{B}_N)_{N\geq 1}$ be two deterministic sequences of $k$-tuples and $\ell$-tuples of $N\times N$ matrices. 
 We assume that
\begin{itemize}
\item  for all $P\in \mathcal{A}_1$, 
$\tau_{\mathbf{A}_N}(P)=\tau(P)+o(1) $,
\item 
for all $P\in \mathcal{I}_1$, 
$\tau_{\mathbf{A}_N}(P)=\frac{1}{N}\tau'(P)+o(N^{-1}),$
\item
for all $P\in\mathcal{A}_2$, 
$\tau_{\mathbf{B}_N}(P)=\tau(P)+o(1),$
\item
for all $P\in \mathcal{I}_2$, 
$\tau_{\mathbf{B}_N}(P)=\frac{1}{N}\tau'(P)+o(N^{-1}).$
\end{itemize}
Let $(U_N)_{N\geq 1}$ be a sequence of Haar distributed unitary  random matrices of size $N$. Then,
\begin{itemize}
\item for all $P\in \mathcal{A}$, the limit
$\EE[ \tau_{\mathbf{A}_N, U_N\mathbf{B}_NU_N^*}(P)]=\tau(P)+o(1)$ exists,
\item for all $P\in \mathcal{I}$, the limit
$\EE[ \tau_{\mathbf{A}_N, U_N\mathbf{B}_NU_N^*}(P)]=\frac{1}{N}\tau'(P)+o(N^{-1})$ exists,
\item $(\mathcal{A}_1,\mathcal{I}_1)$ and $(\mathcal{A}_2,\mathcal{I}_2)$ are free of type $B$ in $(\mathcal{A},\tau,\mathcal{I},\tau')$.
\end{itemize}\label{th:free_type_B}
Furthermore, if $\mathbf{A}_N$ and $\mathbf{B}_N$ are self-adjoint, then,\begin{itemize}
\item for all $P\in \mathcal{A}$, 
$ \tau_{\mathbf{A}_N, U_N\mathbf{B}_NU_N^*}(P)=\tau(P)+o(1)$ almost surely,
\item for all $P\in \mathcal{I}$, 
$\tau_{\mathbf{A}_N, U_N\mathbf{B}_NU_N^*}(P)=\frac{1}{N}\tau'(P)+o(N^{-1})$ almost surely.
\end{itemize}
\end{theorem}
Note that $\tau'$ is not defined on the whole algebra. If we extend $\tau'$ to $\mathcal{A}$, by setting arbitrarily $\tau'(P)=0$ whenever $P\in \mathbb{C}\langle X_{k'+1},\ldots,X_k,Y_{\ell'+1},\ldots,Y_{\ell}\rangle$, then $\mathcal{A}_1$ and $\mathcal{A}_2$ are infinitesimally free with respect to $(\tau,\tau')$, thanks to Lemma~\ref{inf_freeness_from_B_freeness}.
\begin{proof}Thanks to Theorem~\ref{Mainth}, we know that $\mathbb{C}\langle X_1,\ldots,X_{k}\rangle$ and $\mathbb{C}\langle Y_1,\ldots,Y_{\ell}\rangle$ are free up to order $O(N^{-2})$ with respect to $\EE[\tau_{\mathbf{A}_N,U_N \mathbf{B}_NU_N^*}]$.  The freeness of type $B$ and the convergence of $\EE[ \tau_{\mathbf{A}_N, U_N\mathbf{B}_NU_N^*}(P)]$ are obtained by a direct application of Proposition~\ref{Asymptotic_B_freeness}.

Now, let assume that $\mathbf{A}_N$ and $\mathbf{B}_N$ are self-adjoint. We will prove the almost sure convergence by concentration on the unitary group $U(N)$. We will use the following concentration result (which is also an indirect consequence of Corollary 4.4.28 of \cite{anderson2010}).

\begin{theorem}[Corollary 17 of \cite{meckes2013spectral}]
\label{th:concentration_unitary}Let $f$ be a continuous real-valued function on $U(N)$ which, for some constant $C$ and all $U,V\in U(N)$ satisfies
$$|f(U)-f(V)|\leq C\sqrt{\Tr((U-V)(U-V)^*)}.$$
Let $U_N$ be  Haar distributed on $U(N)$. Then we have for all $\delta>0$,
$$\mathbb{P}\left[\left|f(U_N)-\mathbb{E}[f(U_N)]\right|\geq \delta\right]\leq 2e^{-\frac{N\delta^2}{12C^2}}.$$
\end{theorem}

First-of-all, by decomposing into real and imaginary part (which is possible because $\mathbf{A}_N$ and $\mathbf{B}_N$ are self-adjoint), we can assume that $U\mapsto \tau_{\mathbf{A}_N, U\mathbf{B}_NU^*}(P)$ is real.
Let us prove that the continuous real-valued function $f:U\mapsto \tau_{\mathbf{A}_N, U\mathbf{B}_NU^*}(P)$ is Lipschitz. In order to bound $f(U)-f(V)$, we rewrite $f(U)-f(V)$ as a sum of traces by using the swapping trick in order to replace each occurrence of $U$ by $V$ (swapping
one term at a time make appears alternatively $U-V$ or $U^*-V^*$).

The non-commutative H\H{o}lder inequality says that, for any $M_1,\ldots,M_k\in M_N(\mathbb{C})$ and any integers $n_1,\ldots,n_k$ such that $\sum_i 1/(2n_i)=1$, we have
$$\Big|\tr_N(M_1\ldots M_k)\Big|\leq \sqrt[\leftroot{-3}\uproot{3}2n_1]{\tr_N\Big((M_1M_1^*)^{n_1}\Big)}\cdots \sqrt[\leftroot{-3}\uproot{3}2n_k]{\tr_N\Big((M_kM_k^*)^{n_k}\Big)} $$ (see for example \cite[Theorem 2.1.5]{da2018lecture}). Using the non-commutative H\H{o}lder inequality with exponent $n_i=1$ for the term $(U-V)$, and the fact that $\mathbf{A}_N$ and $\mathbf{B}_N$ are bounded in n.c. distribution, 
we conclude that there exists $C>0$ such that
$$|f(U)-f(V)|\leq C\sqrt{\tr_N((U-V)(U-V)^*)}=\frac{C}{\sqrt{N}}\sqrt{\Tr((U-V)(U-V)^*)}.$$
As a consequence, we have
\begin{equation}
\mathbb{P}\left[\left|f(U_N)-\mathbb {E}[f(U_N)]\right|\geq \delta\right]\leq 2e^{-\frac{N^2\delta^2}{12C^2}},\label{eq:Conc Traces}
\end{equation}
and the almost sure convergence (by Borel-Cantelli)
$$ f(U_N)=\mathbb {E}[f(U_N)]+o(1)=\tau(P)+o(1).$$
Now, if $P\in \mathcal{I}$, we have in each term at least one matrix $M_i$ of $\mathcal{I}_1\cup \mathcal{I}_2$  such that 
$$ \Big(\tr_N\big((M_iM_i^*)^2\big)\Big)^{1/4}=O\left(\frac{1}{N^{1/4}}\right)$$
because we assumed that $N\cdot \tr_N\big((M_iM_i^*)^2\big)=O(1)$.
Using the non-commutative H\H{o}lder inequality with exponent $n_i=2$ for this matrix $M_i\in \mathcal{I}_1\cup \mathcal{I}_2$, we conclude that there exists $C>0$ such that
$$| f(U)- f(V)|\leq C\frac{1}{N^{1/4}}\sqrt{\tr_N((U-V)(U-V)^*)},$$
or equivalently
$$|N\cdot f(U)-N\cdot f(V)|\leq CN^{1/4}\sqrt{\Tr((U-V)(U-V)^*)}.$$
We get
$$\mathbb{P}\left[\left|N\cdot f(U_N)-\mathbb {E}[N\cdot f(U_N)]\right|\geq \delta\right]\leq 2e^{-\frac{\sqrt{N}\delta^2}{12C^2}},$$
and the almost sure convergence (by Borel-Cantelli)
$$ N\cdot f(U_N)=N\cdot \mathbb {E}[f(U_N)]+o(1)=\tau'(P)+o(1).$$
\end{proof}

\begin{lemme}\label{inf_freeness_from_B_freeness}
Let $0\leq k'\leq k$ and $0\leq \ell'\leq \ell$. We set $\mathcal{A}:= \mathbb{C}\langle X_1,\ldots,X_k,Y_1,\ldots,Y_{\ell}\rangle$, $\mathcal{A}_1=\mathbb{C}\langle X_1,\ldots,X_k\rangle$ and $\mathcal{A}_2=\mathbb{C}\langle Y_1,\ldots,Y_{\ell}\rangle$. We denote by $\mathcal{I}_1$ the ideal of $\mathcal{A}_1$ generated by $X_1,\ldots,X_{k'}$, by $\mathcal{I}_2$ the ideal of $\mathcal{A}_2$ generated by $Y_1,\ldots,Y_{\ell'}$ and by $\mathcal{I}$ the ideal of $\mathcal{A}$ generated by $\mathcal{I}_1\cup \mathcal{I}_2$. Let $(\mathcal{A},\tau,\tau')$ be an infinitesimal non-commutative probability space such that $\mathcal{I}\subset \ker(\tau)$ and $\mathbb{C}\langle X_{k'+1},\ldots,X_k,Y_{\ell'+1},\ldots,Y_{\ell}\rangle\subset \ker(\tau')$.

We assume that $(\mathcal{A}_1,\mathcal{I}\cap \mathcal{A}_1)$ and $(\mathcal{A}_2,\mathcal{I}\cap \mathcal{A}_2)$ are free of type $B$ in the non-commutative probability space $(\mathcal{A},\tau,\mathcal{I},\tau')$ of type $B$. Then, $\mathcal{A}_1$ and $\mathcal{A}_2$ are infinitesimally free with respect to $(\tau,\tau')$.
\end{lemme}

\begin{proof}The algebras $\mathcal{A}_1$ and $\mathcal{A}_2$ are free with respect to $\tau$. It remains to prove \eqref{condition_inf_freeness}. Let $a_1\in \mathcal{A}_{i_1},\ldots,a_n\in \mathcal{A}_{i_n}$ such that $i_1\neq i_2\neq i_3\neq\cdots$, and  $\tau(a_1)=\cdots=\tau(a_n)=0$.

If $a_1\notin \mathcal{I}\cap \mathcal{A}_{i_1},\ldots,a_n\notin \mathcal{I}\cap\mathcal{A}_{i_n} $, then $a_1,\ldots, a_n$ belongs to $\mathbb{C}\langle X_{k'+1},\ldots,X_k,Y_{\ell'+1},\ldots,Y_{\ell}\rangle$, and so does $a_1\cdots a_n$. As a consequence,
$$\tau' (a_1\cdots a_n)=0=\sum_{i=1}^n\tau(a_1\cdots a_{j-1}\tau' (a_j )a_{j+1}\cdots a_n)
$$because $\mathbb{C}\langle X_{k'+1},\ldots,X_k,Y_{\ell'+1},\ldots,Y_{\ell}\rangle\subset \ker(\tau')$.

If not, there exists $1\leq k \leq n$ such that $a_k \in \mathcal{I}\cap \mathcal{A}_{i_k}$, and we can apply \eqref{type_B_bis} in order to compute
\begin{align*}
\tau' (a_1\cdots a_n)&=\tau(a_1\cdots a_{k-1}\tau' (a_k )a_{k+1}\cdots a_n)\\
&=\sum_{i=1}^n\tau(a_1\cdots a_{j-1}\tau' (a_j )a_{j+1}\cdots a_n)
\end{align*}
because $a_1\cdots a_{j-1}a_{j+1}\cdots a_n \in \mathcal{I}\subset \ker(\tau)$ whenever $j\neq k$.

In all cases, we have
$$\tau' (a_1\cdots a_n)=\sum_{i=1}^n\tau(a_1\cdots a_{j-1}\tau' (a_j )a_{j+1}\cdots a_n)$$
which implies by definition the infinitesimal freeness.
\end{proof}

One direct corollary of Theorem~\ref{th:free_type_B} is the asymptotic cyclic monotone independence of \cite{Collins2018}, stated as follows (see also \cite{arizmendi2021polynomial}).
\begin{theorem}[Theorem 4.1. and Theorem 4.3. of \cite{Collins2018}]
We denote by $\mathcal{A}$ the algebra $\mathbb{C}\langle X_1,\ldots,X_k,Y_1,\ldots,Y_{\ell}\rangle$, by $\mathcal{A}_2$ the algebra $\mathbb{C}\langle Y_1,\ldots,Y_{\ell}\rangle$, by $\mathcal{I}_1$ the ideal of $\mathbb{C}\langle X_1,\ldots,X_{k}\rangle$ generated by $X_1,\ldots,X_{k}$ and by $\mathcal{I}$ the ideal of $\mathcal{A}$ generated by $\mathcal{I}_1$.

Let $(\mathbf{A}_N)_{N\geq 1}$ and $(\mathbf{B}_N)_{N\geq 1}$ be two deterministic sequences of $k$-tuples and $\ell$-tuples of $N\times N$ matrices. 
 We assume that
\begin{itemize}
\item  for all $P\in \mathcal{I}_1$, 
$\tau_{\mathbf{A}_N}(P)=\frac{1}{N}\tau'(P)+o(N^{-1})$,
\item
for all $P\in \mathcal{A}_2$, 
$\tau_{\mathbf{B}_N}(P)=\tau(P)+o(1).$
\end{itemize}
Let $(U_N)_{N\geq 1}$ be a sequence of Haar distributed unitary  random matrices of size $N$. Then,
\begin{itemize}
\item for all $P\in \mathcal{A}$, the limit
$\EE[ \tau_{\mathbf{A}_N, U_N\mathbf{B}_NU_N^*}(P)]=\tau(P)+o(1)$ exists,
\item for all $P\in \mathcal{I}$, the limit
$\EE[ \tau_{\mathbf{A}_N, U_N\mathbf{B}_NU_N^*}(P)]=\frac{1}{N}\tau'(P)+o(N^{-1})$ exists,
\item $(\mathcal{I}_1,\mathcal{A}_2)$ is cyclically monotone with respect to $(\tau',\tau)$.
\end{itemize}\label{th:cyclic_monotone}
Furthermore, if $\mathbf{A}_N$ and $\mathbf{B}_N$ are self-adjoint, then,\begin{itemize}
\item for all $P\in \mathcal{A}$, 
$\tau_{\mathbf{A}_N, U_N\mathbf{B}_NU_N^*}(P)=\tau(P)+o(1)$ almost surely,
\item for all $P\in \mathcal{I}$, 
$\tau_{\mathbf{A}_N, U_N\mathbf{B}_NU_N^*}(P)=\frac{1}{N}\tau'(P)+o(N^{-1})$ almost surely.
\end{itemize}
\end{theorem}
\begin{proof}It suffices to apply Theorem~\ref{th:free_type_B} with $0\leq k'=k$ and $0= \ell'\leq \ell$ in order to get the existence of $\tau$, $\tau'$ and the freeness of type $B$ of $(\mathcal{A}_1,\mathcal{I}_1)$ and $(\mathcal{A}_2,\mathcal{I}_2)$, where $\mathcal{A}_1=\mathbb{C}\langle X_1,\ldots,X_{k}\rangle$ and $\mathcal{I}_1=\{0\}$.

The cyclic monotone independence of $(\mathcal{I}_1,\mathcal{A}_2)$ can be deduced by an application of Proposition~\ref{from_typeB_to_mon}.
\end{proof}
\subsection{Asymptotic conditional freeness}

\begin{definition}Let $\mathbf{A}_N$  be a  $k$-tuple $\mathbf{A}_N=(A_N^{(1)},\ldots,A_N^{(k)})$ of $N\times N$ matrices and a vector $v_N\in \mathbb{C}^N$. The  \emph{non-commutative distribution} (or n.c. distribution) of $\mathbf{A}_N$ with respect to the vector state given by $v_N$ is the linear functional $$ \begin{array}{rcl}\tau_{\mathbf{A}_N}^{v_N}:\mathbb{C}\langle X_1,\ldots,X_k\rangle &\lto& \mathbb{C} \\
P&\lmto& \langle P(A_N^{(1)},\ldots,A_N^{(k)}) v_N, v_N\rangle. \end{array}$$
If $\mathbf{B}_N$  is another  $\ell$-tuple $\mathbf{B}_N=(B_N^{(1)},\ldots,B_N^{(\ell)})$ of $N\times N$ matrices, their \emph{joint distribution} $\tau_{\mathbf{A}_N,\mathbf{B}_N}^{v_N}$ with respect to the vector state given by $v_N$ is the linear functional
$$ \begin{array}{rcl}\tau_{\mathbf{A}_N,\mathbf{B}_N}^{v_N}:\mathbb{C}\langle X_1,\ldots,X_k,Y_1,\ldots,Y_{\ell}\rangle &\lto& \mathbb{C} \\
P&\lmto& \langle P(A_N^{(1)},\ldots,A_N^{(k)},B_N^{(1)},\ldots,B_N^{(\ell)}) v_N, v_N\rangle. \end{array}$$
\end{definition}
Looking in the direction of a deterministic vector, the freeness of type $B$ of previous section yields asympotically  conditional freeness of unitarily invariant random matrices.
\begin{theorem}
\label{th:cond_freeness}
We set $\mathcal{A}:= \mathbb{C}\langle X_1,\ldots,X_k,Y_1,\ldots,Y_{\ell}\rangle$, $\mathcal{A}_1:=\mathbb{C}\langle X_1,\ldots,X_k\rangle$ and $\mathcal{A}_2:=\mathbb{C}\langle Y_1,\ldots,Y_{\ell}\rangle$.

Let $(\mathbf{A}_N)_{N\geq 1}$ and $(\mathbf{B}_N)_{N\geq 1}$ be two deterministic sequences of $k$-tuples and $\ell$-tuples of $N\times N$ matrices, and $(v_N)_{N\geq 1}$ a sequence of unit vectors of $\mathbb{C}^N$.
 We assume that
\begin{itemize}
\item  for all $P\in \mathcal{A}_1$, 
$\tau_{\mathbf{A}_N}(P)=\tau(P)+o(1) $ and $\tau_{\mathbf{A}_N}^{v_N}(P)=\varphi(P)+o(1).$
\item 
for all $P\in \mathcal{A}_2$, 
$\tau_{\mathbf{B}_N}(P)=\tau(P)+o(1)$,
\end{itemize}
Let $(U_N)_{N\geq 1}$ be a sequence of Haar distributed unitary  random matrices of size $N$. Then,
\begin{itemize}
\item for all $P\in \mathcal{A}$, the limit
$\EE[ \tau_{\mathbf{A}_N, U_N\mathbf{B}_NU_N^*}(P)]=\tau(P)+o(1)$ exists,
\item for all $P\in \mathcal{A}$, the limit
$\EE[ \tau^{v_N}_{\mathbf{A}_N, U_N\mathbf{B}_NU_N^*}(P)]=\varphi(P)+o(1)$ exists,
\item $\mathcal{A}_1$ and $\mathcal{A}_2$ are conditionally free with respect to $(\tau,\varphi)$, and $\varphi_{|\mathcal{A}_2}=\tau_{|\mathcal{A}_2}$.
\end{itemize}\label{th:conditionally_free}
Furthermore, if $\mathbf{A}_N$ and $\mathbf{B}_N$ are self-adjoint, then,\begin{itemize}
\item for all $P\in \mathcal{A}$, 
$ \tau_{\mathbf{A}_N, U_N\mathbf{B}_NU_N^*}(P)=\tau(P)+o(1)$ almost surely,
\item for all $P\in \mathcal{A}$, 
$\tau_{\mathbf{A}_N, U_N\mathbf{B}_NU_N^*}^{v_N}(P)=\varphi(P)+o(1)$ almost surely.
\end{itemize}
\end{theorem}
\begin{proof}
By increasing $k$ if necessary, we can assume that the one rank matrix $v_Nv_N^*$ belongs to $\mathbf{A}_N$. However, we need to check that we still have $\tau_{\mathbf{A}_N}(P)=\tau(P)+o(1) $ and $\tau_{\mathbf{A}_N}^{v_N}(P)=\varphi(P)+o(1)$ if we add $v_Nv_N^*$ to $\mathbf{A}_N$  (let say $A_n^{(1)}=v_Nv_N^*$). Denoting by $\mathcal{I}_1$ the ideal of $\mathcal{A}_1$ generated by $X_1$, the rank of $P(\mathbf{A}_N)$ is bounded by $1$ whenever $P\in \mathcal{I}_1$. So, for all $P\in \mathcal{A}_1$, 
$$\tau_{\mathbf{A}_N}(P)=\tau(P)+o(1) $$
with $\tau(P)=0$ if $P\in  \mathcal{I}_1$. Now, if $P$ is a monomial of $\mathcal{I}_1$, $$\tau_{\mathbf{A}_N}^{v_N}(P)=\langle P(\mathbf{A}_N) v_N, v_N\rangle$$ is a product of $\langle M_i(\mathbf{A}_N) v_N, v_N\rangle$ where $M_i$ are monomials which do not contains $X_1$. As consequence, for all $P\in \mathcal{A}_1$, the convergence $$\tau_{\mathbf{A}_N}^{v_N}(P)=\varphi(P)+o(1)$$
is a consequence of the convergence of $\tau_{\mathbf{A}_N}^{v_N}(M_i)$ for monomials $M_i$ which do not contains $X_1$. Because there is no loss of generality, from now on, we assume that $A_N^{(1)}=v_Nv_N^*$.

In order to apply Theorem~\ref{th:free_type_B}, let us check that for all $P\in \mathcal{I}_1$, we have the limit
$\tau_{\mathbf{A}_N}(P)=\frac{1}{N}\tau'(P)+o(N^{-1}).$ By linearity, it suffices to prove it for $P=P_1X_1P_2$. We compute
\begin{align*}
N\cdot \tau_{\mathbf{A}_N}(P_1X_1P_2)&=\Tr(P_1(\mathbf{A}_N)v_Nv_N^*P_2(\mathbf{A}_N))\\
&=\Tr(v_N^*P_2(\mathbf{A}_N)P_1(\mathbf{A}_N)v_N)\\
&=\langle P_2(\mathbf{A}_N)P_1(\mathbf{A}_N) v_N, v_N\rangle\\
&=\tau_{\mathbf{A}_N}^{v_N}(P_2P_1)\\
&=\varphi(P_2P_1)+o(1).
\end{align*}
As a consequence, $N\cdot \tau_{\mathbf{A}_N}(P)$ converges to a limit $\tau'(P)$ for all $P\in \mathcal{I}_1$, and we have
$$\tau'(P_1X_1P_2)=\varphi(P_2P_1).$$
Because $\tau_{\mathbf{A}_N}(P)=\frac{1}{N}\tau'(P)+o(N^{-1})$ for all $P\in \mathcal{I}_1$, we can apply Theorem~\ref{th:free_type_B} with $0\leq k'=1\leq k$ and $0= \ell'\leq \ell$ in order to get the existence of $\tau$, $\tau'$ and the freeness of type $B$ of $(\mathcal{A}_1,\mathcal{I}_1)$ and $(\mathcal{A}_2,\mathcal{I}_2)$, where $\mathcal{I}_1=\{0\}$.

Let $P\in \mathcal{A}$. We have
\begin{align*}\tau_{\mathbf{A}_N, U_N\mathbf{B}_NU_N^*}^{v_N}(P)&=\langle P(\mathbf{A}_N, U_N\mathbf{B}_NU_N^*) v_N, v_N\rangle\\
&=\Tr(v_Nv_N^*P(\mathbf{A}_N, U_N\mathbf{B}_NU_N^*))\\
&=N\cdot \tau_{\mathbf{A}_N, U_N\mathbf{B}_NU_N^*}(X_1P).\end{align*}
As a consequence, the convergence in expectation (resp. almost surely) of $N\cdot \tau_{\mathbf{A}_N, U_N\mathbf{B}_NU_N^*}(X_1P)$ implies the convergence of $\tau_{\mathbf{A}_N, U_N\mathbf{B}_NU_N^*}^{v_N}(P)$ in expectation (resp. almost surely) to
$\tau'(X_1P).$
Extending the domain of $\varphi$ to the whole algebra $\mathcal{A}$ by $\varphi(P):=\tau'(X_1P)$, we just proved that $$\EE[ \tau^{v_N}_{\mathbf{A}_N, U_N\mathbf{B}_NU_N^*}(P)]=\varphi(P)+o(1).$$
The relation $\varphi(P)=\tau'(X_1P)$ and the fact that $\tau(P)=0$ for $P\in \mathcal{I}$ allow us to apply Proposition~\ref{Prop:cfreeness}, yielding to the conditional freeness of $\mathcal{A}_1$ and $\mathcal{A}_2$ with respect to $(\tau,\varphi)$.
\end{proof}
The following particular case, which occurs for example if $\mathbf{A}_N$ have finite ranks, yields asymptotically  monotone independence (see also the recent \cite{collins2022matrix} for another related matrix model).
\begin{corollary}\label{th:mon}We set $\mathcal{A}:= \mathbb{C}\langle X_1,\ldots,X_k,Y_1,\ldots,Y_{\ell}\rangle$, $\mathcal{I}_1$ the ideal of $\mathbb{C}\langle X_1,\ldots,X_{k}\rangle$ generated by $X_1,\ldots,X_{k}$ and $\mathcal{A}_2:=\mathbb{C}\langle Y_1,\ldots,Y_{\ell}\rangle$.

Let $(\mathbf{A}_N)_{N\geq 1}$ and $(\mathbf{B}_N)_{N\geq 1}$ be two deterministic sequences of $k$-tuples and $\ell$-tuples of $N\times N$ matrices, and $(v_N)_{N\geq 1}$ a sequence of unit vectors of $\mathbb{C}^N$.
 We assume that
\begin{itemize}
\item  for all $P\in \mathcal{I}_1$, 
$\tau_{\mathbf{A}_N}(P)=o(1) $ and $\tau_{\mathbf{A}_N}^{v_N}(P)=\varphi(P)+o(1).$
\item 
for all $P\in \mathcal{A}_2$, 
$\tau_{\mathbf{B}_N}(P)=\tau(P)+o(1)$,
\end{itemize}
Let $(U_N)_{N\geq 1}$ be a sequence of Haar distributed unitary  random matrices of size $N$. Then,
\begin{itemize}
\item for all $P\in \mathcal{A}$, the limit
$\EE[ \tau^{v_N}_{\mathbf{A}_N, U_N\mathbf{B}_NU_N^*}(P)]=\varphi(P)+o(1)$ exists,
\item $(\mathcal{I}_1,\mathcal{A}_2)$ is monotonically independent with respect to $\varphi$, and $\varphi_{|\mathcal{A}_2}=\tau_{|\mathcal{A}_2}$.
\end{itemize}
Furthermore, if $\mathbf{A}_N$ and $\mathbf{B}_N$ are self-adjoint, then,\begin{itemize}
\item for all $P\in \mathcal{A}$, 
$\tau_{\mathbf{A}_N, U_N\mathbf{B}_NU_N^*}^{v_N}(P)=\varphi(P)+o(1)$ almost surely.
\end{itemize}
\end{corollary}
\begin{proof}We apply Theorem~\ref{th:conditionally_free} with $\mathcal{I}_1\subset \ker(\tau)$, and conclude about the monotone independence thanks to Proposition~\ref{From_cfree_to_mon}.
\end{proof}
\section{Consequence for the eigenvalues of random matrices \label{sec:eigenvalues}}
\subsection{Proof of Theorem \ref{Th:poly}}
For all $P\in \mathbb{C}\langle X\rangle$, we have
$$\tau_{A_N}(P)= \frac{1}{N}\Tr(P(A_N))=\int_{\mathbb{R}}P\ d\mu_{A_N}\ \text{ and }\ \tau_{A_N}^{v_N}(P)=\langle P(A_N)v_N,v_N\rangle=\int_{\mathbb{R}}P\ d\mu_{A_N}^{v_N}.$$
For all $P\in \mathbb{C}\langle Y\rangle$, we have
$$ \tau_{B_N}(P)=\frac{1}{N}\Tr(P(B_N))=\int_{\mathbb{R}}P\ d\mu_{B_N}.$$
Setting for all $P\in \mathbb{C}\langle X\rangle$,  $$ \tau(P)=\int_{\mathbb{R}}P\ d\mu_1\ \text{ and }\ \varphi(P)=\int_{\mathbb{R}}P\ d\nu_1,$$
and for all $P\in \mathbb{C}\langle Y\rangle$, 
$$ \tau(P)=\int_{\mathbb{R}}P\ d\mu_2,$$
the assumptions of Theorem~\ref{th:conditionally_free} are fulfilled (with $k=\ell =1$) and we get Theorem \ref{Th:poly}. Moreover, we have for all $P\in \mathbb{C}\langle Y\rangle$, 
$$\varphi(P)=\int_{\mathbb{R}}P\ d\mu_2.$$
Theorem~\ref{th:conditionally_free} is also true in the orthogonal case, as explained 
in Section~\ref{Sec:orth}.
\subsection{Proof of Theorem \ref{Th:sum}}
As in the proof of Theorem \ref{Th:poly} in previous section, we can apply Theorem~\ref{th:conditionally_free} by setting for all $P\in \mathbb{C}\langle X\rangle$,  $$ \tau(P)=\int_{\mathbb{R}}P\ d\mu_1\ \text{ and }\ \varphi(P)=\int_{\mathbb{R}}P\ d\nu_1,$$
and for all $P\in \mathbb{C}\langle Y\rangle$, 
$$ \tau(P)=\int_{\mathbb{R}}P\ d\mu_2.$$We get
\begin{itemize}
\item for all $P\in \mathbb{C}\langle X,Y\rangle$, 
$ \tau_{A_N, U_NB_NU_N^*}(P)=\tau(P)+o(1)$ almost surely,
\item for all $P\in \mathbb{C}\langle X,Y\rangle$, 
$\tau_{A_N, U_NB_NU_N^*}^{v_N}(P)=\varphi(P)+o(1)$ almost surely,
\item $\mathbb{C}\langle X\rangle$ and $\mathbb{C}\langle Y\rangle$ are conditionally free with respect to $(\tau,\varphi)$ and $\varphi_{|\mathbb{C}\langle Y\rangle}=\tau_{|\mathbb{C}\langle Y\rangle}$.
\end{itemize}
Note that $\mu_1$ is the distribution of $X$ with respect to $\tau$, $\nu_1$ is the distribution of $X$ with respect to $\varphi$, and $\mu_2$ is the distribution of $Y$ with respect to $\tau$ and with respect to $\varphi$. As a consequence, $\mu_1\boxplus\mu_2$ is the distribution of $X+Y$ with respect to $\tau$ and $\nu_1{\ \!}_{\mu_1\!\!}\boxplus_{\ \!\!\mu_2}\mu_2$ is the distribution of $X+Y$ with respect to $\varphi$ (see Section~\ref{Sec:convolutions}). We conclude because for any $k\geq 0$,
\begin{align*}
&\int_{\mathbb{R}}x^k\ d\mu_{A_N+U_NB_NU_N^*}(x)\\
&=\tau_{A_N,U_NB_NU_N^*}((X+Y)^k)=\tau((X+Y)^k)+o(1)=\int_{\mathbb{R}}x^k\ d\mu_1\boxplus\mu_2(x)
\end{align*}
and
\begin{align*}
&\int_{\mathbb{R}}x^k\ d\mu^{v_N}_{A_N+U_NB_NU_N^*}(x)=\tau^{v_N}_{A_N,U_NB_NU_N^*}((X+Y)^k)=\varphi((X+Y)^k)+o(1)\\
&=\int_{\mathbb{R}}x^k\ d\nu_1\prescript{}{\mu_1}{\boxplus}^{}_{\mu_2} \mu_2(x).\end{align*}
Here again, Theorem~\ref{th:conditionally_free} is also true in the orthogonal case, as explained 
in Section~\ref{Sec:orth}.
\subsection{Proof of Theorem \ref{Th:outliers}}
We follow the approach of Noiry in \cite{Noiry2019}. We apply Theorem \ref{Th:sum} with $v_N$ being one eigenvector associated to the eigenvalue $\theta_N$. In this particular case, $\mu_{A_N}^{v_N}=\delta_{\theta}$ and we get almost surely the following convergence in moments
$$\mu_{A_N+U_NB_NU_N^*}^{v_N}\to \delta_\theta\prescript{}{\mu_1}{\boxplus}^{}_{\mu_2} \mu_2.$$
As $\mu_1$ and $\mu_2$ are compactly supported, it is also the case for $\delta_\theta\prescript{}{\mu_1}{\boxplus}^{}_{\mu_2} \mu_2$, and the convergence in moments is also a weak convergence. In order to conclude, it remains to prove that for any open interval $I \subset \mathbb{R}\setminus supp(\mu_1\boxplus \mu_2)$ containing $\rho$, we have $$\delta_\theta\prescript{}{\mu_1}{\boxplus}^{}_{\mu_2} \mu_2(I)=\frac{1}{\omega'_1(\rho)}.$$Indeed, it would implies the limit
$$\sum_{u\in E_N}\left|\langle u,v_N\rangle \right|^2 =\mu_{A_N+U_NB_NU_N^*}^{v_N}(I)=\delta_\theta\prescript{}{\mu_1}{\boxplus}^{}_{\mu_2} \mu_2(I)+o(1)= \frac{1}{\omega'_1(\rho)}+o(1)$$
and consequently the existence of an eigenvalue of $A_N+U_NB_NU_N^*$ in any neighbourhood of $\rho$ for sufficiently large $N$.

The value of $\delta_\theta\prescript{}{\mu_1}{\boxplus}^{}_{\mu_2} \mu_2(\{\rho\})$ is given by the residue of $$G_{\delta_\theta\prescript{}{\mu_1}{\boxplus}^{}_{\mu_2} \mu_2}(z)=\frac{1}{\omega_1(z)-\theta}=\frac{1}{\omega_1(z)-\omega_1(\rho)}$$
at $\rho$, which is
$$\lim_{z\to \rho}\frac{z-\rho}{\omega_1(z)-\omega_1(\rho)}=\frac{1}{\omega'_1(\rho)}.$$
As $\omega_1$ extends meromorphically  to $\mathbb{C}\setminus supp(\mu_1\boxplus\mu_2),$ maps $\C_+$ to $\C_+$ and is  real-valued on $\mathbb{R}\setminus supp(\mu_1\boxplus\mu_2)$, it is injective on $I$. In particular, $\rho$ is the only pole  in $I$ of $$G_{\delta_\theta\prescript{}{\mu_1}{\boxplus}^{}_{\mu_2} \mu_2}(z)=\frac{1}{\omega_1(z)-\omega_1(\rho)}.$$Note that $G_{\delta_\theta\prescript{}{\mu_1}{\boxplus}^{}_{\mu_2} \mu_2}$ is analytic on $ I\setminus \{\rho\}$, which means that $\delta_\theta\prescript{}{\mu_1}{\boxplus}^{}_{\mu_2} \mu_2$ is absolutely continuous with respect to the Lebesgue measure on $ I\setminus \{\rho\}$. Because $\omega_1$ is real-valued on $I\setminus \{\rho\}$, the Stieltjes-Perron formula gives a vanishing density
$$-\frac{1}{\pi}\lim_{\varepsilon \to 0^+}\Im G_{\delta_\theta\prescript{}{\mu_1}{\boxplus}^{}_{\mu_2} \mu_2}(x+i\varepsilon)=0,$$
for $x\in I\setminus \{\rho\}$, which allows to conclude
$$\delta_\theta\prescript{}{\mu_1}{\boxplus}^{}_{\mu_2} \mu_2(I)=\delta_\theta\prescript{}{\mu_1}{\boxplus}^{}_{\mu_2} \mu_2(\{\rho\})=\frac{1}{\omega'_1(\rho)}.$$
\section{Proof of Theorem~\ref{Mainth}}\label{proof_of_Mainth}

The first result of asymptotic freeness  for large unitarily invariant matrices  was given by Voiculescu \cite{MR1094052}, relying on Gaussian computation for complex Ginibre matrices, polar decomposition and concentration. One of the  first direct  proof (using explicit moments computation)   is given by Biane in~\cite[Section 9]{Biane1998}  and later by Collins \cite{MR1959915}. The former proof, closely related to~\cite{Xu1997}, has been extended to the orthogonal and symplectic invariance by Collins and \'Sniady in~\cite{Collins2004}. Detailed  arguments can now be found in~\cite[Lecture 23]{Nica2006} and in~\cite[Chapter 3]{Levy2011}.

\subsection{A first proof with the Weingarten function}
\label{sec:Weingarten}
This first proof relies  almost solely on \cite[Lecture 23]{Nica2006} and a second-order expansion of the Weingarten function found in \cite{MR1959915,Collins2004}. See also \cite{WCIntro} for a recent introduction to Weingarten calculus with useful references.

\vspace{0,5 cm}

\noindent For  $P_1,\ldots,P_n\in \mathbb{C}\langle X_1,\ldots,X_k\rangle$ and $Q_1,\ldots,Q_n\in \mathbb{C}\langle X_1,\ldots,X_\ell\rangle$, let 
$$ A^{(k)}= P_k(\mathbf{A}_N) \text{  and  }  B^{(k)}= P_k(\mathbf{B}_N), \text{ for } 1\le k\le n.$$
When $M_1,\ldots,M_n\in M_N(\C)$ are $n$ matrices and $\sigma \in S_n $  is a permutation, define 
$$\tr_\sigma(M_i,1\le i\le n)= \prod_{(i_1\ldots i_k) \text{ cycle of } \sigma} \tr_N(M_{i_1}\ldots M_{i_k}).$$ 
It is enough to prove the identity of Theorem \ref{proof_of_Mainth} for $P$ of the form $$P=P_1(X) Q_1(Y)\ldots P_n(X) Q_n(Y).$$  Therefor, we recall  first the computation of
\begin{equation}
\mathbb{E}[\tau_{\mathbf{A}_N, U_N\mathbf{B}_NU_N^*}(P)]=\mathbb{E}\left[\tr_N\left( UA^{(1)}U^*B^{(1)}\cdots UA^{(n)}U^*B^{(n)}\right)\right],\label{eq:UInt}
\end{equation}
where $U_N$ is a Haar distributed random variable.  For any $\sigma \in S_n,$
$$\Wg_N(\sigma) = \EE(U_{1,1}\overline{U}_{1,\sigma(1)}\ldots U_{n,n}\overline{U}_{n,\sigma(n)})$$
is called the \emph{Weingarten function}. Remarkably \cite{Collins2004}, all moments in $U$ can be computed\footnote{The formula below seems to have first appeared in the physics paper \cite{Samuel_1980}.} thanks to $\Wg_N:$  for any $i,j,i',j'\in \{1,\ldots,N\}^n,$
\begin{equation}
 \EE(U_{i_1,j_1}\overline{U}_{i'_1,j'_1}\ldots U_{i_n,j_n}\overline{U}_{i'_n,j'_n})=\sum \Wg_N(\alpha^{-1}\beta), \label{eq:WeingU}
\end{equation}
where the sum is over all $\alpha,\beta\in S_n$ with $i'_k=i_{\alpha(k)},j'_k=j_{\beta(k)}$ for all $1\le k\le n.$ Expanding the traces and using \eqref{eq:WeingU}   imply\footnote{see  ~\cite[Lecture 23]{Nica2006} (p. 386) or ~\cite[Section 3.3]{Levy2011}} that   the integral \eqref{eq:UInt}  is equal to 
\begin{equation}
\sum_{\alpha,\beta\in S_n} \Wg_N(\alpha^{-1}\beta) N^{\#\alpha+\# \beta^{-1} \gamma_n -1} \tr_\alpha(A^{(k)},1\le k \le n)\tr_{\beta^{-1}\gamma_n} (B^{(k)},1\le k \le n), \label{eq: N Formula Mixed Moment}
\end{equation}
where $\gamma_n$  is the full cycle $(1\ldots n)$ and for any permutation $\sigma$, $\# \sigma$ denotes the number of cycles of $\sigma$. Now it is known, see for instance Corollary 2.7 of \cite{Collins2004},  that for all $\sigma\in  S_n,$
$$\mathrm{Moeb}_N(\sigma):=N^{2n-\#\sigma }\Wg_N(\sigma)= \mathrm{Moeb}(\sigma)+ N^{-2} \Mp(\sigma)+O(N^{-4}) ,$$
where $\M$ and $\Mp$ are  functions on $S_n$. For any permutations $\alpha,\beta,\gamma\in  S_n,$ consider 
 $$\df(\alpha,\beta,\gamma)=2n-(\#\alpha+\# \alpha^{-1}\beta+\# \beta^{-1}\gamma)+\#\gamma$$
 so  that, noting that $\# \gamma_n=1,$ \eqref{eq: N Formula Mixed Moment}  reads
\begin{equation}
\sum_{\alpha,\beta\in S_n} \M_N(\alpha^{-1}\beta) N^{-\df(\alpha,\beta,\gamma)} \tr_\alpha(A^{(k)},1\le k \le n)\tr_{\beta^{-1}\gamma_n} (B^{(k)},1\le k \le n). \label{eq:Mixed Moments defect}
\end{equation}
We conclude with a suitable triangular inequality. Setting $ |\sigma| = n- \#\sigma$ for all $\sigma\in S_n$
defines a distance on $S_n$. It is the distance of the Cayley graph of  $S_n$  generated by all transpositions. For $\alpha,\beta\in S_n,$ denote 
by 
\begin{equation}
\df(\alpha,\beta)= |\alpha|+ |\alpha^{-1}\beta|-|\beta|\label{eq:Defect SymGp}
\end{equation}
the defect of $(1,\alpha, \beta)$ from being a geodesic triangle. It is well known that 
\begin{equation}
\df(\alpha,\beta)\in 2\N \label{eq:Even Defect SymGp}.
\end{equation}
Note that for any $\alpha,\beta,\gamma\in S_n,$
$$\df(\alpha,\beta,\gamma)=\df(\alpha,\beta)+\df(\beta,\gamma)\in 2 \N$$
so  that our expectation \eqref{eq:UInt} is 
\begin{align*}
&\sum_{\alpha,\beta\in S_n: \df(\alpha,\beta,\gamma)=0} \mathrm{Moeb}(\alpha^{-1}\beta) \tr_\alpha(A^{(k)},1\le k \le n)\tr_{\beta^{-1}\gamma_n} (B^{(k)},1\le k \le n)\\
&+ N^{-2} R_2(A,B)+ O(N^{-4}),
\end{align*}
where  $R_2(A,B)$ is a  polynomial of $(\tr_{\alpha}(A^{(k)},1\le k\le n),\tr_{\alpha}(B^{(k)},1\le k\le n))_{\alpha\in S_n}.$  
Setting  $\alpha\prec\beta$ whenever $\df(\alpha,\beta)=0,$  define 
\begin{equation}
\kappa_\beta(A)=  \sum_{\alpha\prec \beta}\mathrm{Moeb}(\alpha^{-1}\beta) \tr_\alpha(A), \label{eq:DefFreeCum}
\end{equation}
and 
$$\kappa^{(2)}_\beta(A)= \sum_{\alpha\prec \beta} \Mp(\alpha^{-1}\beta) \tr_\alpha(A)+ \sum_{\alpha: \df(\alpha,\beta)=2} \mathrm{Moeb}(\alpha^{-1}\beta) \tr_\alpha(A). $$
We can then write  \eqref{eq:UInt} as 
\begin{equation}
\sum_{\beta\prec \gamma_n} \kappa_{\beta}(A) \tr_{\beta^{-1} \gamma_n}(B^{(k)},1\le k\le n) + N^{-2} R_2(A,B)+O(N^{-4}),\label{eq:AlternatedProduct Expansion2}
\end{equation}
where
\begin{align*}
R_2(A,B)&=  \sum_{\beta\prec \gamma_n} \kappa^{(2)}_{\beta}(A) \tr_{\beta^{-1} \gamma_n}(B^{(k)},1\le k\le n)\\
&+ \sum_{\beta: \df(\beta,\gamma)=2} \kappa_\beta(A)  \tr_{\beta^{-1} \gamma_n}(B^{(k)},1\le k\le n).
\end{align*}
To conclude it is enough to prove that  for any fixed $N,$
\begin{equation}
\sum_{\beta\prec \gamma_n} \kappa_{\beta}(A) \tr_{\beta^{-1} \gamma_n}(B^{(k)},1\le k\le n) =\tau_{\mathbf{A}_N}\star \tau_{\mathbf{B}_N}(P),\label{eq:Free alternated P}
\end{equation}
where $P$ is  $P_1(X) Q_1(Y)\ldots P_n(X) Q_n(Y)$ and  $P_{k},Q_k(Y)$ denote   $Q_k(Y_1,\ldots, Y_l)$ for  $1\le k\le n$. Indeed, this last identity  yields

$$\mathbb{E}\left[\tau_{\mathbf{A}_N,U_N\mathbf{B}_NU_N^*}(P)\right]=\tau_{\mathbf{A}_N}\star \tau_{\mathbf{B}_N}(P)+N^{-2}\Psi_{\mathbf{A}_N,\mathbf{B}_N}(P)+O(N^{-4}),$$
where
\begin{align*}
\Psi_{\mathbf{A}_N,\mathbf{B}_N}(P)&=  \sum_{\beta\prec \gamma_n} \kappa^{(2)}_{\beta}(A) \tr_{\beta^{-1} \gamma_n}(B^{(k)},1\le k\le n)\\
&+ \sum_{\beta: \df(\beta,\gamma)=2} \kappa_\beta(A)  \tr_{\beta^{-1} \gamma_n}(B^{(k)},1\le k\le n).
\end{align*}




By induction on the degree of $P_k,Q_k, 1\le k\le n,$ it is enough to assume that 
$$\tr_N(P_k(\mathbf{A}_N))=\tr_N(P_k(\mathbf{B}_N))=0, \,\,\text{ for } 1\le k\le n$$ and to show that $\sum_{\beta\prec \gamma_n} \kappa_{\beta}(A) \tr_{\beta^{-1} \gamma_n}(B^{(k)},1\le k\le n)=0.$   
Expanding the summand with \eqref{eq:DefFreeCum} leads to
\begin{equation*}
\sum_{\alpha\prec\beta\prec \gamma_n} \M(\alpha)\tr_{\alpha^{-1}\beta}(A) \tr_{\beta^{-1} \gamma_n}(B^{(k)},1\le k\le n).\tag{*} \label{eq:MixedMoment3}
\end{equation*}
Now for any $\alpha,\beta\in S_n$ with $\alpha\prec\beta\prec \gamma_n,$ 
$|\alpha|+|\alpha^{-1}\beta|+|\beta^{-1} \gamma_n|=|\gamma_n|=n-1,$ so that 
$$\min(|\alpha^{-1}\beta|,  |\beta^{-1} \gamma_n|)\le \frac {n-1}2.$$
Since a product  of $k$ transpositions has at least $n-2k$  fixed points, it follows that $\alpha^{-1}\beta$ or $\beta^{-1} \gamma_n$ has a fixed point. But since $P_k,Q_k, 1\le k\le n$ are centered, for any permutation $\sigma \in S_n$ with at least one fixed point
$$\tr_\sigma(A^{(k)},1\le k \le n)=\tr_\sigma(B^{(k)},1\le k \le n)=0$$
and the expression \eqref{eq:MixedMoment3} vanishes.

\begin{rem} Alternatively, \eqref{eq:Free alternated P}  can be proved identifying the poset appearing in \eqref{eq:DefFreeCum}. Consider the poset $[1,\gamma_n]=\{\sigma\in S_n:  \sigma\prec \gamma_n\}$ endowed with the order $\prec$ and the poset $NC(n)$ of non-crossing partitions of $\{1,\ldots,n\}$, where for any $\pi_1,\pi_2\in NC(n),$ $\pi_1\preccurlyeq \pi_2$
 if each  block of $\pi_1$ is included in a block of $\pi_2$.   It can be shown \cite{BIANE199741} that  mapping a permutation $\sigma\in [1,\gamma_n]$ to the partition  of $\{1,\ldots,n\}$ into  orbits of $\sigma$ defines a bijection 
 between the two posets, mapping $\sigma^{-1}\gamma_n$ to the Kreweras complement $K(\pi_{\sigma})$ of $\pi_{\sigma}$.  Let us set  
 \[ \tau_{\pi_\sigma}(A)= \tr_\sigma(A)\text{ and } \kappa_{\pi_\sigma}(A)= \prod_{i_1<\ldots < i_l \text{ block of }\pi_\sigma}\kappa_{\# b}(A^{(i)},i\in b),\forall \beta \in [1,\gamma_n] \]
 where   $(\kappa_m)_{m\ge 1}$ are the free-cumulants associated to $\tau_{\mathbf{A}_N}.$ 
 Using \eqref{eq:AlternatedProduct Expansion2} taking the identity matrix in place of $B$, it follows that 
 \[\kappa_{\beta}(A)=\kappa_{\pi_\beta}(A),\forall \beta \in [1,\gamma_n].\]
The claim \eqref{eq:Free alternated P} is now equivalent to the standard relation 
\[ \tau_{\mathbf{A}_N}\star \tau_{\mathbf{B}_N}(P)=\sum_{\pi \in NC(n) } \kappa_{\pi}(A) \tau_{K(\pi)}(B).  \]

%
\end{rem}
%

\subsection{A second proof with matricial cumulants}
\label{section:matricial_cumulants}
We shall use here the notion of matricial cumulants first introduced by Capitaine et Casalis~\cite{CapitaineCum,Capitaine2008} and generalized in \cite{gabriel2015combinatorial1,gabriel2015combinatorial2,gabriel2015combinatorial3}.  After recalling their definition following \cite{gabriel2015combinatorial1,gabriel2015combinatorial2,gabriel2015combinatorial3}, it shall  enable us to give two proofs of  Theorem \ref{Mainth}, relying only on  the Schur-Weyl duality, as well as an order  expansion  in $N^{-2}$ for the joint non-commutative distribution  of $(\mathbf{A}_N,\mathbf{B}_N).$ 
\subsubsection{Schur-Weyl duality and matricial cumulants}




The symmetric group $S_n$ acts on the space of tensors  $(\C^{N})^{\ts n}$  by permuting them: 
$$ \sigma. v_1\ts\ldots \ts v_n =v_{\sigma^{-1}(1)}\ts \ldots\ts v_{\sigma^{-1}(n)}, \hspace{1 cm}\forall \sigma\in S_n, v_1,\ldots, v_n\in \C^N.$$
When $\sigma\in S_n$ is a permutation, we shall  use abusively the same symbol  $\sigma $ for the endomorphism  acting on tensors. The matrix of the  adjoint of the endomorphism $\sigma$ in the canonical basis of $(\C^{N})^{\ts n}$ is denoted by $\sigma^t$. Let  $\Tr_{n,N}$ be the non-normalised trace on $M_N(\C)^{\ts n}$. These notions relate to our previous setting in the following way.

\begin{lemme} When $M_1,\ldots,M_n\in M_N(\C)$ are $n$ matrices and $\sigma\in S_n,$
$$\Tr_{n,N}( M_1\ts M_2\ts \ldots \ts M_n \circ \sigma^t )=  \prod_{(i_1\ldots i_k) \text{ cycle of } \sigma} \Tr_N(M_{i_1}\ldots M_{i_k}).$$
\end{lemme}
\noindent In particular, using the above  notation, 
\begin{equation}
\tr_\sigma(M_i, 1\le i\le n)= N^{-\#\sigma} Tr_{n,N}( M_1\ts M_2\ts \ldots \ts M_n \circ \sigma^t).
\end{equation}
Recall that $ |\sigma| = n- \#\sigma $ defines a distance on $S_n$. The Schur-Weyl duality can be understood as follows. 
\begin{theorem}\label{th:Schur-Weyl} When $M_1,\ldots,M_n\in M_N(\C)$ are $n$ matrices, $U$ is a Haar distributed random unitary matrice, there exist  coefficients   $(\kappa^N_\sigma(\mathbf{M}))_{\sigma\in S_n}$ such that
\begin{align}
\label{eq:Schur-Weyl}
\EE_U[ (UM_1U^*)\ts \ldots \ts (UM_nU^*)]=\sum_{\sigma\in S_n} N^{-|\sigma|} \kappa^N_\sigma(\mathbf{M}) \sigma  \in \mathrm{End}((\C^{N})^{\ts n}).
\end{align}
These coefficients are called the \emph{matricial cumulants}  of the tuple $\mathbf{M} = (M_1,\ldots,M_n)$ \cite{gabriel2015combinatorial1,gabriel2015combinatorial2,gabriel2015combinatorial3}. They are uniquely defined as long as $N\ge n.$
\end{theorem}

The cumulants are not uniquely defined when $N <n$. For any such $n$ and any family $(a_\sigma)_{\sigma \in S_n}$, the equality $\kappa_{\sigma}^{N}(\mathbf{M}) = a_\sigma$ for any $\sigma \in S_n$ means that Equation \eqref{eq:Schur-Weyl} holds when one  replaces $\kappa_{\sigma}^{N}(\mathbf{M})$ by $a_\sigma$. 

When $M_1,\ldots, M_n$ are random, the expectation of the matricial cumulants are obtained by considering the decomposition of $\EE[ (UM_1U^*)\ts \ldots \ts (UM_nU^*)]$, with $U$ a Haar distributed random matrix independent from $M_1,\ldots,M_n$: 
\begin{align}
\label{eq:Schur-Weyl-random}
\EE[ (UM_1U^*)\ts \ldots \ts (UM_nU^*)]=\sum_{\sigma\in S_n} N^{-|\sigma|} \EE\left[\kappa^N_\sigma(\mathbf{M})\right] \sigma.
\end{align}

This definition  was given in this form in \cite{Capitaine2008,gabriel2015combinatorial1,gabriel2015combinatorial2}, together with generalisation to other compact groups \cite{gabriel2015combinatorial1,gabriel2015combinatorial2}. Note that the  first definition of \cite{CapitaineCum} in the unitary case uses the Weingarten function: in Section \ref{sec:computation}, we will see that both definitions are equivalent. Yet, a benefit from the formulation above is that it follows easily from it that mixed cumulants in independent unitary invariant random matrices vanish in expectation. 

\begin{theorem}
\label{th:mixed_vanish_cumulants}
Let $\mathbf{A}$ and $\mathbf{B}$ be two deterministic $k$-tuple and $\ell$-tuple of $N\times N$ matrices and $U$ be a Haar distributed unitary matrix. For any  $n$ and any $C_1,\ldots,C_n \in \mathbf{A} \cup U\mathbf{B}U^*$, 
\begin{enumerate}
\item for any $\sigma\in S_n$, if there exist two indices $i\neq j$ in the same cycle of $\sigma$ such that $C_i \in \mathbf{A}$ and $C_j \in U\mathbf{B}U^*$, then $\mathbb{E}[\kappa^N_{\sigma}(C_1,\ldots,C_n)]=0$, 
\item for any $0\leq m \leq n$, any $\sigma_1\in S_m$, any $\sigma_2\in S_{n-m}$, if $C_i \in \mathbf{A}$ for $i\in {1,\ldots,m}$ and $C_{j} \in U\mathbf{B}U^*$, i.e. $C_j = UB^{(c_j)}U^*$for $j\in {m+1,\ldots,n}$ then 
$$\mathbb{E}[\kappa^{N}_{\sigma_1\otimes \sigma_2}(C_1,\ldots,C_n)]= \kappa^N_{\sigma_1}(C_1,\ldots,C_m)\kappa^N_{\sigma_2}\left(B^{(c_{m+1})},\ldots,B^{(c_{n})}\right).$$
\end{enumerate} 
\end{theorem}
 
Note that, using a concentration argument (Theorem \ref{th:concentration_unitary}), one can control the mixed matricial cumulants of $(\mathbf{A}, U\mathbf{B}U^*)$ and not only their expectations. In order to be self-contained, we provide the proof of Theorem \ref{th:mixed_vanish_cumulants}. 

\begin{proof}
We assume that $\mathbf{A}$ and $\mathbf{B}$ are both composed of one matrix, namely $A$ and $B$, the general case can be easily deduced using the same arguments. 
Assume that $C_1,\ldots,C_m = A$ and $C_{m+1}, \ldots, C_n = UBU^*$. Let $V$ be a Haar distributed unitary matrix, independent from $U$. The tensor $\EE[ (VC_1V^*)\ts \ldots \ts (VC_nV^*)]$ is equal to
\begin{align*}
\EE_{U,V}\left[ (VAV^*)^{\ts m} \ts (UBU^*)^{\ts n-m}\right] &= \EE_{V}\left[ (VAV^*)^{\ts m} \right] \ts \EE_{U} \left[(UBU^*)^{\ts n-m}\right]. 
\end{align*}
From the definition of the matricial cumulants, $\EE[ (VC_1V^*)\ts \ldots \ts (VC_nV^*)]$ is: 
\begin{align*}
\sum_{\sigma_1 \in S_m,\sigma_2 \in S_{n-m}} N^{-|\sigma_1|} N^{-|\sigma_2|} \kappa^N_{\sigma_1}(A)\kappa^N_{\sigma_2}(B)\left(\sigma_1 \otimes \sigma_2\right). 
\end{align*}
The endomorphism $\sigma_1\otimes \sigma_2$ corresponds to the action of the permutation in $S_n$, denoted also by $\sigma_1\otimes \sigma_2$, such that: 
\begin{enumerate}
\item $\sigma_1\otimes \sigma_2(i) = \sigma_1(i)$ for any $i\in \{1,\ldots,m\}$,
\item $\sigma_1\otimes \sigma_2(m+i) = m+\sigma_2(i)$ for any $i\in \{1,\ldots,n-m\}$.
\end{enumerate}
Since $|\sigma_1|+|\sigma_2|=|\sigma_1\otimes \sigma_2|$, we obtain the equality: 
$$\sum_{\sigma\in S_n}\!\! N^{-|\sigma|} \mathbb{E}\left[\kappa_\sigma^{N}(A,\ldots,A,B,\ldots,B)\right] \!\sigma = \!\! \sum_{\sigma_1 \in S_m,\sigma_2 \in S_{n-m}}  \!\! \!\! N^{-|\sigma_1\otimes \sigma_2|} \kappa^N_{\sigma_1}(A)\kappa^N_{\sigma_2}(B) \!\left(\sigma_1 \otimes \sigma_2\right),$$ 
which allows us to conclude. 
\end{proof}

These matricial cumulants are natural approximations of free cumulants\footnote{See  \cite{gabriel2015combinatorial1,gabriel2015combinatorial2,gabriel2015combinatorial3}  where the expansion as a function of $\frac 1 N$ is understood thanks to the notion of  cumulants of higher order for  sequence of matrices  with $\tr_\sigma(M_i,1\leq i\leq n)$  having  moments with a $\frac{1}{N}$ expansion, thereby generalizing the notion of infinitesimal freeness  to higher orders.}.  Indeed, considering  the trace $\Tr_{n,N}$ of each member of Equality (\ref{eq:Schur-Weyl}) composed with a permutation $\sigma\in S_n$ leads to the following analog of moment-cumulant relation. Recall Equation (\ref{eq:Defect SymGp}) where we defined the defect $\df(\alpha,\beta)$ of $\alpha$ not being on a geodesic between the identity $1$ and $\beta$. Let  $\mathbf{M} = (M_1,\ldots,M_n)$ be a $n$-tuple of matrices in  $M_N(\C)$. 

\begin{prop} \label{Prop:Matrix Moment Cum} For all $\sigma\in S_n,$
\begin{align}
\label{eq:moment_cumulant}
\tr_\sigma(\mathbf{M} )=  \sum_{\alpha \in S_n} N^{-\df(\alpha,\sigma)} \kappa^N_{\alpha}(\mathbf{M} ).
\end{align}
\end{prop}

\begin{proof}
Let $U$ be a Haar distributed random matrix. Using the cyclic property of the trace, the definition of $\tr_\sigma$, and the definitions of the matricial cumulants, we get: 
\begin{align*}
\tr_\sigma(\mathbf{M}) = \mathbb{E}[tr_\sigma(U\mathbf{M} U^*)] &=N^{-\#\sigma} Tr_{n,N}( \mathbb{E}[ UM_1U^*\ts UM_2U^*\ts \ldots UM_nU^*] \circ \sigma^t)\\
&= \sum_{\alpha\in S_n} N^{-\#\sigma -|\alpha|} \kappa^N_\alpha(\mathbf{M}) Tr_{n,N}(\alpha \circ \sigma^t).
\end{align*}
For any permutation $\beta$, the endomorphism $\beta^t$ corresponds to the action of the permutation $\beta^{-1}$ on $(\mathbb{C}^N)^{\otimes n}$ and $Tr_{n,N}(\beta)=N^{\#\beta}$. Hence using also the cyclic property of the trace and its invariance by transposition, $Tr_{n,N}(\alpha \circ \sigma^t) = N^{\# (\alpha^{-1} \sigma)}$, which yields: 
$$\tr_\sigma(\mathbf{M}) = \sum_{\alpha\in S_n} N^{-\#\sigma -|\alpha| + \#(\alpha^{-1} \sigma)} \kappa^N_\alpha(\mathbf{M}).$$
Using the definition of the defect $\df(\alpha,\sigma)$ (Equation \ref{eq:Defect SymGp}), we can conclude.\end{proof}

\begin{rem}\label{sketch-proof}
One can now provide a sketch of proof of Theorem \ref{Mainth}, when $\mathbf{A}_N$ and $\mathbf{B}_N$ are both composed of one matrix, namely $A_N$ and $B_N$. Again, the general case can be easily deduced using the same arguments.  

\noindent \em{1}- The defect $\df(\alpha,\sigma)\in 2\mathbb{N}$: Equation (\ref{eq:moment_cumulant}) yields that $$\tr_\sigma(\mathbf{M})=  \sum_{\alpha,\df(\alpha,\sigma)=0} \kappa^N_{\alpha}(\mathbf{M}) + O(N^{-2}). 
$$
\noindent 2- It implies that, up to order $O(N^{-2})$, the matricial cumulants are equal to the {\em free cumulants} $\kappa_\sigma (\mathbf{M})$ defined by induction by the relation 
 $$\tr_\sigma(\mathbf{M})=  \sum_{\alpha,\df(\alpha,\sigma)=0} \kappa_{\alpha}(\mathbf{M}).  
$$
\noindent 3- From Equation (\ref{eq:moment_cumulant}), up to order $O(N^{-2})$, the expectation of the mixed moments of $C_1,\ldots,C_n \in \{A_N,U_NB_NU_N^*\}$ are: 
\begin{align*}
\mathbb{E}\left[\tr_\sigma (C_1,\ldots,C_n)\right] &=  \sum_{\alpha,\df(\alpha,\sigma) = 0} \mathbb{E}\left[ \kappa^N_{\alpha}(C_1,\ldots,C_n)\right] . 
\end{align*}
\noindent 4- Assume for sake of simplicity, that $C_1,\ldots,C_m = A_N$ and $C_{m+1},\ldots,C_{n} = U_NB_NU_N^*$, then from Theorem \ref{th:mixed_vanish_cumulants}, 
\begin{align*}
\mathbb{E}\left[\tr_\sigma (C_1,\ldots,C_n)\right] &=  \sum_{{\alpha_1\in S_m, \alpha_2 \in S_{n-m},}\atop{\df(\alpha_1\otimes \alpha_2,\sigma) = 0}} \kappa^N_{\alpha_1} (A_N) \kappa^N_{\alpha_2} (B_N) + O(N^{-2})
\end{align*}
and from the second point, 
\begin{align*}
\mathbb{E}\left[\tr_\sigma (C_1,\ldots,C_n)\right] &=  \sum_{{\alpha_1\in S_m, \alpha_2 \in S_{n-m},}\atop{\df(\alpha_1\otimes \alpha_2,\sigma) = 0}} \kappa_{\alpha_1} (A_N) \kappa_{\alpha_2} (B_N) + O(N^{-2})
\end{align*}
\noindent 5- The first order in the r.h.s. is the expectation of the same mixed moment of $A_N$ and $B_N$, if they were free. This allows us to conclude.
\end{rem}

We now provide the details of the proof and, doing so, we go beyond the $N^{-2}$ order, define some higher order matricial cumulants for fixed $N$, and provide an expansion at any order of the mixed moments of $\mathbf{A}_N$ and $U_N\mathbf{B}_NU_N^*$ (see Theorem \ref{Mainth-general} which implies Theorem \ref{Mainth}). 

\begin{corollary}  \label{Coro:Moment Cumulant Matrix} There exist some multilinear maps $(\kappa_{\sigma}^{(2g)}(\cdot ))_{\sigma\in S_n, g\ge 0}$, which are linear combinations with  integer  coefficients \footnote{ the  sign of the $\tr_\alpha$ coefficient in $\kappa_{\sigma,g}$ is  $(-1)^{|\sigma|+|\alpha|}.$ }  in the set of observables $\{ \tr_{\sigma}( \cdot),$ $\sigma\in S_n \}$  such that  for $N$ large enough, for any n-tuple of matrices $\mathbf{M} = (M_1,\ldots,M_n) \in M_N(\mathbb{C})$, 
\begin{align*}
\kappa^N_{\sigma}(\mathbf{M})= \sum_{g\ge 0} N^{-2 g} \kappa_{\sigma}^{(2g)}(\mathbf{M}). 
\end{align*}
\end{corollary}

The new higher order matricial cumulants $\kappa_{\sigma}^{(2g)}(\mathbf{M})$ defined above should not be mistaken for the higher order matrix cumulants $k^{(g)}_\sigma((\mathbf{M}_N)_N)$ defined in \cite{gabriel2015combinatorial2}. Indeed, the latter were only defined for sequence of matrices $(\mathbf{M}_N)_N$ whose moments have an expansion in powers of $N^{-1}$.

\begin{proof}
For any $2k\in \mathbb{N}$, we denote by $G_{2k}$  the matrix indexed by $S_n \times S_n$ such that $G_{2k}(\sigma,\alpha) = \delta_{\df(\alpha, \sigma)=2k}$. From Equation (\ref{eq:moment_cumulant}), it remains to prove that the inverse of the matrix $\sum_{k \geq 0}\frac{1}{N^{2k}} G_{2k}$ can be expanded as $\sum_g \frac{1}{N^{2g}}C_{2g}$ where $C_{2g}$ is an integer valued matrix for any $g\geq 0$. Indeed, this would imply the desired result with 
\begin{align}
\label{eq:kappa_2g}
\kappa^{(2g)}_\sigma (M)= \sum_{\alpha\in S_n} C_{2g}(\sigma,\alpha) \tr_{\alpha} (\mathbf{M}).
\end{align}
Using the equation 
$\left(\sum_{k \geq n}\frac{1}{N^{2k}} G_{2k}\right)\left(\sum_g \frac{1}{N^{2g}}C_{2g}\right) = \mathrm{Id}_{S_n},$ and collecting the coefficients in $N^{-2g}$, the matrices $C_{2g}$ can be defined by induction since: 
\begin{align}
\label{eq:induction_kappa_2g}
G_{0} C_{0}  = \mathrm{Id}_{S_n}, \quad \sum_{i=0}^{g} G_{2i} C_{2(g-i)} = 0 \end{align}
for any $g\geq 1$. In particular, $C_{0} = G_{0}^{-1}$ and $C_{2g} = - \sum_{i=1}^{g} C_{0} G_{2i} C_{2(g-i)}$ for $g\geq 1$. 

Since $G_{0}$ can be written as the sum of the identity matrix and an integer-valued nilpotent matrix, its inverse $C_0$, which is given by the finite Taylor expansion of the inversion near the identity is also integer-valued. By induction, this implies that all the matrices $C_{2g}$ are integer-valued. Besides, the defect is bounded in $S_n$, hence the matrix $G_{2k}$ vanishes for large $k$. In particular, $\sup_k{ ||G_{2k}||} = c_G < \infty$. This allows to obtain a bound of the form 
\begin{equation}
|| C_{2g} || \leq 2^{g} c_{G}^g || G_0^{-1} ||^{g+1}\label{eq:Bound Moment to Cumulant  2g transition}
\end{equation}
If  $N$ is large enough, the sum $\sum_g \frac{1}{N^{2g}}C_{2g}$ is well-defined. This allows us to conclude. 
\end{proof}

We shall simply write $\kappa_{\sigma}(\mathbf{M})$ for $\kappa_{\sigma}^{(0)}(\mathbf{M})$. From the previous proof, $(\kappa_{\sigma}(\mathbf{M}))_\sigma$ are the free cumulants of $\mathbf{M}$ as defined in the second step of Remark \ref{sketch-proof}. The higher-order cumulants can be computed also by induction. 
\begin{lemme}\label{lem:Higher order cum}  For all $\sigma\in S_n$, for any n-tuple of matrices $\mathbf{M} = (M_1,\ldots,M_n) \in M_N(\mathbb{C})$, 
$$ \tr_\sigma(\mathbf{M})=\sum_{\alpha\prec \sigma} \kappa_\alpha(\mathbf{M}),$$
whereas  for any $g\ge 1,$ 
$$\sum_{\alpha\in S_n, \df(\alpha,\sigma)\le 2g} \kappa_{\alpha}^{(2g - \df(\alpha,\sigma))}(\mathbf{M})   =0. $$
\end{lemme}

\begin{proof} 
Consider the vectors $\tr_{\cdot}(\mathbf{M}) = (\tr_{\sigma}(\mathbf{M}))_{\sigma \in S_n}$ and $\kappa_{\cdot}^{(2g)}(\mathbf{M}) = (\kappa_{\sigma}^{(2g)}(\mathbf{M}))_{\sigma \in S_n}$. 
From Equations (\ref{eq:induction_kappa_2g}) and (\ref{eq:kappa_2g}), 
$$ \delta_{g=0} \tr_{\cdot}(\mathbf{M}) = \sum_{i=0}^{g} G_{2i} C_{2(g-i)} \tr_{\cdot}(\mathbf{M})  = \sum_{i=0}^{g} G_{2i} \kappa_{\sigma}^{(2(g-i))}(\mathbf{M}). $$
This yields, the two desired equations. 
\end{proof}

These relations can be used in order to compute the higher order matricial cumulants. In particular, one can provide formulae relating these matricial cumulants, the M\"obius function and the Weingarten function (see Section \ref{sec:computation} for the higher order). Using the bijection between the poset $NC(n)$ and  $\{\alpha\in S_n: \alpha\prec (1\ldots n)\}, $ it is classical to show the following expression for the $0$-order matricial cumulants.
\begin{lemme} \label{lem:free_cumulants} It follows from the bijection between  the posets $NC(n)$ and $ \{ \alpha\in S_n: \alpha\prec (1\ldots n)\}$  that 
$$\kappa_\sigma(\mathbf{M})= \prod_{ (i_1\ldots i_k) \text{ cycle of }\sigma}\kappa_{k}(M_{i_1},\ldots,M_{i_k}),$$
where  for any matrices $C_1,\ldots, C_k\in M_N(\C),$ $\kappa_k(C_{1},\ldots,C_{k})$  denotes the free cumulants of $(C_1,\ldots,C_k)$ for $\tr_N$, that is 
$$\kappa_k(C_1,\ldots, C_k)= \sum_{\pi \in [0_k, 1_k]_{NC(k)}}  \M_{NC(n)}(\pi, 1_k) \prod_{ i_1< \ldots <i_l \text{ block  of } \pi }  \tr_N(C_{i_1}\ldots C_{i_l}). $$
\end{lemme}

\subsubsection{Mixed moments, matricial cumulants and proof of Theorem \ref{Mainth}}

Using matricial cumulants, we can compute a $N^{-2}$ expansion of the mixed moments in $\mathbf{A}_N$ and $\mathbf{B}_N$ (Theorem \ref{Mainth-general}) which implies Theorem \ref{Mainth}. When $(\mathbf{C}_N)_N$ is a sequence of $k$-tuples of matrices of size $N$, let us set for any $n\ge 1,$  
\[\|\mathbf{C}\|_{*,n}=\sup_{N,\sigma\in S_n,1\le i_1,\ldots i_n\le k}|\tr_{\sigma}(C_{i_1}\ldots C_{i_n})|. \]

\begin{theorem}\label{Mainth-general}
Let $(\mathbf{A}_N)_{N\geq 1}$ and $(\mathbf{B}_N)_{N\geq 1}$ be two deterministic sequences of $k$-tuples and $\ell$-tuples of $N\times N$ matrices, each of them being bounded in n.c. distribution.  Let $(U_N)_{N\geq 1}$ be a sequence of Haar distributed unitary  random matrices of size $N$. For any non-commutative polynomial $P$, 
$$\EE[\tau_{\mathbf{A}_N, U_N\mathbf{B}_NU_N^*}(P)]=\tau_{\mathbf{A}_N}\star \tau_{\mathbf{B}_N}(P) + \sum_{g=1}^{k} N^{-2g } \tau_{\mathbf{A}_N}\star_g\tau_{\mathbf{B}_N}(P) + N^{-2k-2}R_{k+1}(\mathbf{A},\mathbf{B}), $$ 
 where $|R_{k+1}(\mathbf{A},\mathbf{B})|= O(\sum_{a+b\le \mathrm{deg}(P)}\|A\|_{*,a}\|B\|_{*,b})$ and $\tau_{\mathbf{A}_N}\star_g \tau_{\mathbf{B}_N}(P)$ can be computed in the two following ways. 
 
 If  $P=A^{(1)}B^{(1)}\cdots A^{(n)}B^{(n)},$ with $ A^{(k)}= P_k(\mathbf{A}_N)$  and  $B^{(k)}= P_k(\mathbf{B}_N), \text{ for } 1\le k\le n,$ then $\mu_{\mathbf{A}_N}\star_g\mu_{\mathbf{B}_N}(P)$ is equal to: 
 \begin{align}
\sum_{\alpha\in S_n: \df(\alpha,(1\ldots n))\le2g} \kappa^{(2g-\df(\alpha,(1\ldots n)))}_\alpha(A^{(k)},1\le k\le n)  \tr_{\alpha^{-1} (1\ldots n)}(B^{(k)},1\le k\le n) 
 \end{align}
 or equivalently to 
\begin{align}  
\sum_{ \substack{a,b\ge 0,\alpha,\beta\in S_n \\ a+b+\df(\alpha, \alpha\circ \beta)+\df(\alpha\circ\beta,(1,\ldots,n))= 2g   } } \kappa_\alpha^{(a)}(A^{(k)},1\le k\le n)\kappa_\beta^{(b)}(B^{(k)},1\le k\le n).
 \end{align}

 
\end{theorem} 

\begin{proof}
We can assume that $\mathbf{A}_N$ and $\mathbf{B}_N$ are both composed of one matrix, namely $A_N$ and $B_N$, and we only need to prove the expansion of $\EE[\tau_{A_N, U_N B_N U_N^*}(P)]$ and the value of $\tau_{A_N}\star_g \tau_{B_N}(P)$ for $P=(AB)^n$. The general case can be easily deduced using the same arguments.  

Note that $\mathbb{E}[\tau_{A_N,U_NB_NU_N^*}(P)] = \mathbb{E}\left[\tr\left( ((V_NA_NV_N^*) (U_N B_N U_N^*))^n \right)\right]$ where $V_N$ is an independent Haar distributed matrix. Hence 
\begin{align*}
\mathbb{E}[\tau_{A_N,U_NB_NU_N^*}(P)] = 
\frac{1}{N} \Tr_{n,N}\left[ \mathbb{E}\left[(V_N A_N V_N^*)^{\otimes n}\right] \circ  \mathbb{E}\left[ (U_N B_N U_N^*)^{\otimes n}\right] \circ (1,\ldots,n)^t \right] 
\end{align*}
Recall the definition of the matricial cumulants (Theorem \ref{th:Schur-Weyl}) and their expansion (Corollary \ref{Coro:Moment Cumulant Matrix}) in powers  of $N^{-1}$. Recall also that for any permutation $\beta$, the endomorphism $\beta^t$ corresponds to the action of the permutation $\beta^{-1}$ on $(\mathbb{C}^N)^{\otimes n}$ and $Tr_{n,N}(\beta)=N^{\#\beta}$. Then, if one writes in terms of cumulants:  
\begin{enumerate}
\item the tensor $ \mathbb{E}\left[(V_N A_N V_N^*)^{\otimes n}\right]$, one gets that $\mathbb{E}[\tau_{A_N,U_NB_NU_N^*}(P)] $ is equal to: 
\begin{align}
\label{eq:expansion_mixed_moments}
\sum_{g\geq 0} N^{-2g} \sum_{\alpha\in S_n | \df(\alpha,(1,\ldots,n)) \leq 2g} \kappa^{(2g-\df(\alpha,(1,\ldots,n)) )}_{\alpha}(A_N) \tr_{\alpha^{-1} (1,\ldots,n)} (B_N),
\end{align}

\item both tensors $\mathbb{E}\left[(V_N A_N V_N^*)^{\otimes n} \right] $ and $\mathbb{E}\left[ (U_N B_N U_N^*)^{\otimes n}\right] $,  one gets that the moment $\mathbb{E}[\tau_{A_N,U_NB_NU_N^*}(P)] $ is equal to: 
\begin{align*}
\sum_{g\geq 0}N^{-2g} \sum_{ \substack{a,b\ge 0,\alpha,\beta\in S_n \\ a+b+\df(\alpha, \alpha\circ \beta)+\df(\alpha\circ\beta,(1,\ldots,n))= 2g   } } \kappa_\alpha^{(a)}(A_N)\kappa_\beta^{(b)}(B_N).
\end{align*}
\end{enumerate}

Using Corollary  \ref{Coro:Moment Cumulant Matrix} together with Equations \eqref{eq:kappa_2g} and \eqref{eq:Bound Moment to Cumulant  2g transition}  which allows us to bound the higher order cumulants,
we obtain the expansion stated in the theorem and the formulae for $\mu_{A_N}\star_g\mu_{B_N}(P)$. 
At last, note that when $g=0$, the first term in the expansion in Equation (\ref{eq:expansion_mixed_moments}) is simply 
$$\sum_{\alpha \prec (1,\ldots,n)} \kappa^{(0)}_{\alpha}(A_N) \tr_{\alpha^{-1} (1,\ldots,n)} (B_N).$$ 
Using Lemma \ref{lem:free_cumulants}, we can conclude that the first term is simply given by the free convolution $\tau_{A_N} \star \tau_{B_N} (P)$. 
\end{proof}

The bounded coefficients $\tau_{\mathbf{A}_N}\star_g\tau_{\mathbf{B}_N}(P)$ do depend on $N$. In concrete cases, we might be given sequences of matrices having  distribution with a $\tfrac{1}{N}$ expansion up to order $K\in \mathbb{N}$ in the following sense:
$$\tau_{\mathbf{A}_N}(P)=\sum_{k=0}^K\frac{1}{N^k}\tau_{\mathbf{A}}^{(k)}(P)+o\left(N^{-K}\right).$$
From such expansions, it is possible to expand each term $\tau_{\mathbf{A}_N}\star_g\tau_{\mathbf{B}_N}(P)$ as a series in $\tfrac{1}{N}$, and Theorem~\ref{Mainth-general} yields an expansion
$$\EE[\tau_{\mathbf{A}_N, U_N\mathbf{B}_NU_N^*}(P)]=\sum_{k=0}^K\frac{1}{N^k}\tau^{(k)}(P)+o\left(N^{-K}\right)$$
with coefficients $\tau^{(k)}(P)$ which do not depend on $N$.
We do not provide here the formulae for computing the coefficients $\tau^{(k)}(P)$ from the expansions of $\mathbf{A}_N$ and $\mathbf{B}_N$, as it has already been done independently in \cite{gabriel2015combinatorial2} and \cite{Borot2021}. More precisely, the existence of this expansion is ensured by~\cite[Theorem 8.5]{gabriel2015combinatorial2} and the computation is possible thanks to the notion of \emph{freeness up to order $K$ of fluctuations}. This expansion is also a consequence of~\cite[Theorem 4.28]{Borot2021} and the computation is possible thanks to the notion of \emph{$(\tfrac{K}{2},1)$-freeness}. The difference of Theorem~\ref{Mainth-general} from \cite[Theorem 8.5]{gabriel2015combinatorial2} and \cite[Theorem 4.28]{Borot2021} is that we only assume the boundedness in n.c. distribution of $(\mathbf{A}_N)_{N\geq 1}$ and $(\mathbf{B}_N)_{N\geq 1}$ instead of assuming the expansion of $\tau_{\mathbf{A}_N}$ and $\tau_{\mathbf{B}_N}$ in $\tfrac{1}{N}$.

\begin{rem}
\label{rem:computation_mixed_moments}
Note that from the computations in the previous proof, the mixed moment $\mathbb{E}[\tau_{A_N,U_NB_NU_N^*}((A_NB_N)^n)]$ is equal to $\sum_{\alpha\in S_n} N^{-\df(\alpha,(1,\ldots,n))} \kappa^N_{\alpha} (A_N)  \tr_{\alpha^{-1} (1,\ldots,n)} (B_N)$ or $\sum_{\alpha,\beta\in S_n} N^{-\df(\alpha,\alpha\circ \beta) - \df(\alpha\circ \beta , (1,\ldots,n))} \kappa^N_{\alpha}(A_N) \kappa^N_{\beta}(B_N)$. This shows how one can use the mixed matricial cumulants in order to compute mixed moments (see Example \ref{ex:computation_mixed_moments}). 
\end{rem}

\subsubsection{Computations of the higher order matricial cumulants}
\label{sec:computation}
We illustrate the notion of matricial cumulants by first showing how one can compute them  explicitly  when $n=2$. As a corollary of this, we obtain the higher order matricial cumulants. Then, we compute a mixed moment of $A$ and $UBU^*$ using these matricial cumulants. At last, we provide the relation between these higer order matricial cumulants and the Weingarten function.

\begin{ex} 
\label{ex:computation_mixed_moments}
Consider $A \in  M_N(\C)$  fixed. Using Proposition \ref{Prop:Matrix Moment Cum}, 
\begin{align*}
\tr_{1_2}(A,A)&=  \kappa_{1_2}^N(A,A) + N^{-2}  \kappa_{(1\,2)}^N(A,A), \\
\tr_{(1\, 2)} (A,A)&= \kappa_{1_2}^N(A,A) +  \kappa_{(1\,2)}^N(A,A). 
\end{align*}
Therefore, 
\begin{align*}\kappa_{1_{2}}^{N}(A,A)&=\frac{1}{1-N^{-2}}\tr_{N}(A)^{2}-\frac{N^{-2}}{1-N^{-2}}tr_N(A^{2})\\
 \kappa_{(1\,2)}^N(A,A)&=\frac{1}{1-N^{-2}}(\tr_N(A^2)- \tr_N(A)^2).
 \end{align*}
Using Corollary \ref{Coro:Moment Cumulant Matrix}, it follows that
$$\kappa_{1_2}^{(2g)}(A,A) = \left\{ \begin{array}{lc}  \tr_N(A)^2 & \text{ if } g=0,\\   \tr_N(A)^2-\tr_N(A^2)&\text{ if } g\ge 1,\end{array}\right.$$
whereas 
$$\kappa_{(1\,2)}^{(2g)}(A,A)= \tr_N(A^2)-\tr_N(A)^2, \hspace{0,5 cm}\forall g\ge 0.$$

Consider now $A,B\in M_N(\C)$ and $\tilde{B}=UBU^*$ where $U$ is a Haar unitary matrix. Let us compute $\EE [\tr(A\tilde{B}A\tilde{B})]$ using the matricial cumulants (see also Remark \ref{rem:computation_mixed_moments}): 
 \begin{align*}
\EE[\tr(A\tilde{B}A\tilde{B})]&= \frac{1}{N} \Tr_{2,N}\left[\left( \kappa_{1_2}^N(A,A) 1_2 +N^{-1}\kappa_{(1\, 2)}^N(A,A) (1\,2) \right)B^{\ts 2} (1,2)^t\right] \\
& =  \kappa_{1_2}^N(A,A) \tr_N(B^2) +  \kappa_{(1,2)}^N(A,A) \tr_N(B)^2. 
\end{align*}
Note that it can also be expressed as: 
 \begin{align*} \kappa_{1_2}(A,A) \tr_N(B^2) +  \kappa_{(1,2)}(A,A) \tr_N(B)^2 -\frac{N^{-2}}{1-N^{-2}}  \kappa_{(1\,2)}( A,A)\kappa_{(1\,2)}(B,B). 
 \end{align*}
 \end{ex}

In Lemma \ref{lem:free_cumulants}, we identified the $0$-order matricial cumulants as the free cumulants. This identification yields a way to compute explicitly them. All higher order cumulants $\kappa^{(g)}$ can also be expressed thanks to Lemma \ref{lem:Higher order cum}. Let us give an example with $g=1.$
\begin{lemme} For any $\sigma\in S_n,$ any n-tuple of matrices $\mathbf{M} = (M_1,\ldots,M_n) \in M_N(\mathbb{C})$, 
\begin{align*}
 \kappa^{(2)}_\sigma(\mathbf{M})&=-\sum_{ \substack{\alpha,\beta\in S_n\\ \df(\beta,\alpha)=2,  \alpha\prec \sigma}}\M(\alpha^{-1}\sigma)  \kappa_{\beta}(\mathbf{M})  \\
 &=- \sum_{ \substack{\alpha,\beta,\gamma\in S_n\\ \gamma\prec \beta, \df(\beta,\alpha)=2,  \alpha\prec \sigma}} \M(\alpha^{-1}\sigma) \M(\gamma^{-1}\beta) \tr_\gamma(\mathbf{M}).
\end{align*}
\end{lemme}

\begin{proof}
The function $\M$ is related to the Moebius function for the order $\prec$. More precisely, two families $(a_\sigma)_{\sigma \in S_n}$ and $(b_\sigma)_{\sigma \in S_n}$ satisfy  $
a_\sigma = \sum_{\alpha \in S_n, \alpha \prec \sigma} b_\alpha$ for any $\sigma \in S_n$, if and only if
\begin{align*}
b_\sigma = \sum_{\alpha \in S_n, \alpha \prec \sigma} \M(\alpha^{-1} \sigma) a_\alpha 
\end{align*}
holds for any $\sigma \in S_n$. From Lemma \ref{lem:Higher order cum}, we know that: 
$$
\sum_{\alpha\prec \sigma} \kappa_\alpha(\mathbf{M}) = \tr_\sigma(\mathbf{M}), \quad\quad \sum_{\alpha \prec \sigma } \kappa_{\alpha}^{(2)} (\mathbf{M}) = \zeta_\sigma,
$$
where $\zeta_\sigma= \sum_{\alpha\in S_n, \alpha\neq \sigma \atop \df(\alpha,\sigma)= 2} \kappa_{\alpha}^{}(\mathbf{M})$. This yields the two formulas for $\kappa_\sigma^{(2)}(\mathbf{M})$. 
\end{proof}

Actually, matricial cumulants relate to the Weingarten function. Recall that, in Section \ref{sec:Weingarten}, we saw that the function $\M_N(\cdot ) := N^{n+|\cdot|}\Wg_N(\cdot)$ can be expanded in power of $N^{-2}$ (see ~\cite{Collins2004}): for any $\sigma \in S_n$,   
$\M_N(\sigma)= \sum_{g\ge 0 } \M^{(g)}(\sigma)N^{-2g}.$
The matricial cumulants can then be expressed using $\M_N$ and $ \M^{(g)}$. 

\begin{lemme}For any $\sigma\in S_n,$ any n-tuple of matrices $\mathbf{M} = (M_1,\ldots,M_n) \in M_N(\mathbb{C})$, 
$$\kappa_\sigma^{N}(\mathbf{M})= \sum_{\alpha\in S_n} N^{-\df(\alpha,\sigma)}\M_N(\alpha) \tr_{\alpha^{-1}\sigma}(\mathbf{M}). $$
Hence, for any $g\geq 0$, 
\begin{equation}
\kappa_\sigma^{(g)}(\mathbf{M}) = \sum_{\alpha\in S_n: \df(\alpha,\sigma)\le 2g} \M^{(g-\df(\alpha,\sigma)/2)}(\alpha) \tr_{\alpha^{-1}\sigma}(\mathbf{M}).
\end{equation}
\end{lemme}
\begin{proof}
 Equation \eqref{eq:WeingU} can be expressed as:  
 \begin{align*}
\mathbb{E}\left[ U^{\otimes n}\otimes (U^*)^{\otimes n} \ \right] = \sum_{\alpha,\beta \in S_n} \Wg_N(\alpha \beta) \tau \circ (\alpha \otimes \beta), 
\end{align*}
where $\tau(i) = n+i $ mod. $2n$. This implies that: 
\begin{align*}
\mathbb{E}\left[UM_1U^{*}\otimes \ldots \otimes UM_nU^{*} \right] = \sum_{\alpha,\beta \in S_n} \Wg_N(\alpha \beta) Tr_{n,N}(M_1\otimes \ldots \otimes M_n \circ \beta) \alpha. 
\end{align*}
Using the definition of the matricial cumulants, 
\begin{align*}
\kappa_\sigma^{N}(\mathbf{M}) &= N^{|\sigma|} \sum_{\beta \in S_n} \Wg_N(\sigma \beta) Tr_{n,N}(M_1\otimes \ldots \otimes M_n \circ \beta), \\
&= \sum_{\alpha\in S_n} N^{-\df(\alpha,\sigma)}\M_N(\alpha) \tr_{\alpha^{-1}\sigma}(\mathbf{M})
\end{align*}
which allows us to conclude for the first equation. The second equation is obtained by using the expansion of $\M_N$ and identifying the coefficients in the expansion.
\end{proof}

\section{The orthogonal case}\label{Sec:orth}
In this section, we indicate which results remain true and which ones do not when Haar distributed unitary  random matrices are replaced by Haar distributed orthogonal random matrices. Let us denote by $O(N)$ the orthogonal group of real $N\times N$ matrices $U_N$ such that $U_NU_N^t=I_N$. We consider also the subgroup $SO(N)=\{U_N\in O(N) : \det(U_N)=1\}$.

The first proof of Theorem~\ref{Mainth} can be adapted to the orthogonal case. It has been done by Collins and \'{S}niady in~\cite{Collins2004} (see also \cite[Chapter 3]{Levy2011}). The main difference is that the Weingarten function is converging up to order $O(N^{-1})$. It yields to the following asymptotic freeness.

\begin{theorem}\label{Mainth_orth}
Let $(\mathbf{A}_N)_{N\geq 1}$ and $(\mathbf{B}_N)_{N\geq 1}$ be two deterministic sequences of $k$-tuples and $\ell$-tuples of real symmetric $N\times N$ matrices.  Let $(U_N)_{N\geq 1}$ be a sequence of orthogonal random matrices of size $N$, uniformly distributed in $SO(N)$, or $O(N)$.

Then, for all $P\in \mathbb{C}\langle X_1,\ldots,X_{k},Y_1,\ldots,Y_{\ell}\rangle$,
$$\EE[\tau_{\mathbf{A}_N, U_N\mathbf{B}_NU_N^*}(P)]=\tau_{\mathbf{A}_N}\star \tau_{\mathbf{B}_N}(P)+N^{-1}R(\mathbf{A}_N,\mathbf{B}_N),$$
with $|R(\mathbf{A}_N,\mathbf{B}_N)|=O(\sum_{a+b\le \mathrm{deg}(P)}\|\mathbf{A}_N\|_{*,a}\|\mathbf{B}_N\|_{*,b})$. In other words, $\mathbb{C}\langle X_1,\ldots,X_{k}\rangle$ and $\mathbb{C}\langle Y_1,\ldots,Y_{\ell}\rangle$ are free up to order $O(N^{-1})$ with respect to $\EE[\tau_{\mathbf{A}_N, U_N\mathbf{B}_NU_N^*}]$ whenever $(\mathbf{A}_N)_{N\geq 1}$ and $(\mathbf{B}_N)_{N\geq 1}$ are  bounded in n.c. distribution. 
\end{theorem}
Note that the rate of convergence is not sufficient here to conclude about infinitesimal freeness. In fact, the invariance in law by orthogonal conjugation does not imply infinitesimal freeness, as shown by Mingo in \cite{Mingo2018}. As well, the rate of convergence seems not to be sufficient to conclude about freeness of type $B$. However, note that if $\|\mathbf{A}_N\|_{*,a}\|\mathbf{B}_N\|_{*,b}$ is $O(N^{-1})$, we have $N^{-1}R(\mathbf{A}_N,\mathbf{B}_N)=O(N^{-2})$. This simple observation allows to get the asymptotic freeness of type $B$ as follows.

\begin{theorem}In Theorem~\ref{th:asymp_freeness_type_B}, the matrices $(U_N)_{N\geq 1}$ can be taken uniformly distributed in $O(N)$ or in $SO(N)$ whenever $(\mathbf{A}_N)_{N\geq 1}$ and $(\mathbf{B}_N)_{N\geq 1}$ are real symmetric matrices.
\label{th:asymp_freeness_type_B_orth}
\end{theorem}

\begin{proof}The proof follows the one of Theorem~\ref{th:asymp_freeness_type_B_orth}. It is not possible to apply directly Proposition~\ref{Asymptotic_B_freeness} here because the freeness is not of order $O(N^{-2})$. However, we have the rate of convergence $O(N^{-2})$ on $\mathcal{I}$, which yields the convergence of $\EE[\tau_{\mathbf{A}_N, U_N\mathbf{B}_NU_N^*}]$ to $\tau$ and $\tau'$, and the freeness of type $B$ as explained in Remark~\ref{Asymptotic_B_freeness_ortho}.

The almost sure convergence is a consequence of the concentration phenomenon. The analogue of Theorem~\ref{th:concentration_unitary} is true for the special orthogonal group $SO(N)$ (see \cite[Corollary 4.4.28]{anderson2010}). Unfortunately, this concentration fail to apply to $O(N)$ because it is not connected. Following \cite{meckes2013concentration}, we condition on the two possible values of $\det(U_N)$. For this purpose, it is useful to consider the uniform measure on the homogeneous space $SO^{-}(N)=\{U_N\in O(N) : \det(U_N)=-1\}$ (it is the invariant measure under the action of $SO(N)$). Theorem~\ref{Mainth_orth} and Theorem~\ref{th:asymp_freeness_type_B_orth} remain true if $(U_N)_{N\geq 1}$ is uniformly distributed in $SO^{-}(N)$ because the concentration phenomenon holds for $SO^{-}(N)$ (see Proposition 2.2 of \cite{meckes2013concentration}). The almost sure result for $O(N)$ is obtained by conditioning on the two possible values of the determinant, and using the result for $SO(N)$ and $SO^{-}(N)$.
\end{proof}
As for the unitary case, we obtain (cyclic) monotone independence and conditional independence as consequences of Theorem~\ref{th:asymp_freeness_type_B_orth}.
\begin{theorem}
In Theorem~\ref{th:cyclic_monotone}, Theorem~\ref{th:conditionally_free} and Corollary~\ref{th:mon}, the matrices $(U_N)_{N\geq 1}$ can be taken uniformly distributed in $O(N)$ or in $SO(N)$ whenever $(\mathbf{A}_N)_{N\geq 1}$ and $(\mathbf{B}_N)_{N\geq 1}$ are real symmetric matrices, and $v_N\in \mathbb{R}^N$.
\end{theorem}

\bibliographystyle{alpha}
\bibliography{biblio}
\end{document}